\documentclass[12pt]{iopart}
\expandafter\let\csname equation*\endcsname\relax
\expandafter\let\csname endequation*\endcsname\relax 
\usepackage[latin1]{inputenc}
\usepackage{amsmath,amsthm,enumitem,datetime,xfrac,fix-cm,color}
\usepackage{amsfonts,amssymb,mathtools,stmaryrd,dsfont,tikz,subcaption}
\usepackage{algorithm,algorithmic,graphicx}
\definecolor{linkColor}{rgb}{0.5, 0.5, 0.5}
\usepackage[colorlinks, allcolors=linkColor, bookmarks=True]{hyperref}
\usetikzlibrary{positioning}
\DeclareGraphicsExtensions{.png, .pdf}
\numberwithin{equation}{section}
\newtheorem{definition}{Definition}[section]\newtheorem{theorem}{Theorem}[section]
\newtheorem{lemma}{Lemma}[section]\newtheorem{example}{Example}[section]
\newtheorem{remark}{Remark}[section]
\makeatletter\let\c@definition\c@theorem\let\c@definition\c@theorem\makeatother
\makeatletter\let\c@lemma\c@theorem\let\c@lemma\c@theorem\makeatother
\makeatletter\let\c@example\c@theorem\let\c@example\c@theorem\makeatother
\makeatletter\let\c@remark\c@theorem\let\c@remark\c@theorem\makeatother
\newtheorem*{theorem*}{Theorem}

\newcommand{\F}[1]{\mathds{ #1 }}\newcommand{\C}[1]{\mathcal{ #1 }}

\newcommand{\MTwo}[4]{\begin{pmatrix} #1 & #2 \\ #3 & #4 \end{pmatrix}}
\newcommand{\norm}[1]{\left\lVert #1 \right\rVert}\renewcommand{\sfrac}[2]{^{#1}\!\!/\!_{#2}}
\renewcommand{\IP}[2]{\left\langle #1,#2 \right\rangle}\newcommand{\1}[0]{\F{1}}
\newcommand{\op}[1]{\operatorname{#1}}

\newcommand{\splitln}[4]{\left\{\begin{array}{cc} #1 & #2 \\ #3 & #4\end{array}\right.}

\DeclareMathOperator{\st}{\;s.t.\;}\renewcommand{\epsilon}{\varepsilon}
\DeclareMathOperator{\union}{\cup}

\renewcommand{\vec}{\mathbf}
\DeclareMathOperator*{\argmin}{argmin}\DeclareMathOperator{\TV}{TV}\DeclareMathOperator{\DTV}{DTV}

\usepackage[normalem]{ulem}
\newcommand{\oldtext}[1]{}
\newcommand{\newtext}[1]{{#1}}

\begin{document}
	\author{Robert Tovey$^1$, Martin Benning$^2$, Christoph Brune$^3$, Marinus J. Lagerwerf$^4$, Sean M. Collins$^5$, Rowan K. Leary$^5$, Paul A. Midgley$^5$, Carola-Bibiane Sch\"onlieb$^1$}
	\address{1 Centre for Mathematical Sciences, University of Cambridge}
	\address{\newtext{2 School of Mathematical Sciences, Queen Mary University of London}}
	\address{3 University of Twente}
	\address{4 Centrum Wiskunde \& Informatica}
	\address{5 Department of Materials Science and Metallurgy, University of Cambridge}
	\date{\longdate\today}
	\title{Directional Sinogram Inpainting for Limited Angle Tomography}
	\begin{abstract}
		In this paper we propose a new joint model for the reconstruction of tomography data under limited angle sampling regimes. In many applications of Tomography, e.g. Electron Microscopy and Mammography, physical limitations on acquisition lead to regions of data which cannot be sampled. Depending on the severity of the restriction, reconstructions can contain severe, characteristic, artefacts. Our model aims to address these artefacts by inpainting the missing data simultaneously with the reconstruction. Numerically, this problem naturally evolves to require the minimisation of a non-convex and non-smooth functional so we review recent work in this topic and extend results to fit an alternating (block) descent framework. \oldtext{We illustrate the effectiveness of this approach with numerical experiments on two synthetic datasets and one Electron Microscopy dataset.} \newtext{We perform numerical experiments on two synthetic datasets and one Electron Microscopy dataset. Our results show consistently that the joint inpainting and reconstruction framework can recover cleaner and more accurate structural information than the current state of the art methods.}
	\end{abstract}
%	\noindent{\it tomography, limited angle tomography, inpainting, anisotropy, variational, nonconvex, optimization}
	
	\maketitle
	
	\section{Introduction}
	\subsection{Problem Formulation}
	Many applications in materials science and medical imaging rely on the X-ray transform as a mathematical model for performing 3D volume reconstructions of a sample from 2D data. We shall refer to any modality using the X-ray transform forward model as X-ray tomography. This encompasses a huge range of applications including Positron Emission Tomography (PET) in front-line medical imaging \cite{Ehrhardt2015}, Transmission Electron Microscopy (TEM) in materials or biological research \cite{Leary2013,Kubel2005,Zhao2013}, and X-ray Computed Tomography (CT) which enjoys success across many fields \cite{Kalender2006,Cnudde2013}. 
	
	\newtext{The limited angle problem is common in X-ray tomography, for instance in TEM \cite{Kawase2007} and Mammography \cite{Zhang2006}, and is caused by a particular limited data scenario. Algebraically, we search for approximate solutions to the inverse problem
	\[ \text{Given data } b\text{ find optimal pair } (u,v) \text{ such that }Sv = b, \C Ru = v\]
	where $\C R$ is the X-ray transform to be defined in \eqref{radon}, $S$ represents the limited angle sub-sampling pattern described in Figure \ref{limitedangle_eg}.	Typically, limited angle problems can occur due to having a large sample or because equipment does not allow the sample to be fully rotated. 
	Mathematically, microlocal analysis can be used to categorise the limited angle problem and characterise artefacts that occur. 	Viewed through the Fourier slice theorem, it becomes clear that the Fourier coefficients of $u$ are partitioned into those `visible' in $b$ and those contained in a `missing wedge' \cite{Quinto1993}. These coefficients are referred to respectively as the visible and invisible singularities of $u$. The limited angle problem then is both a denoising and inpainting inverse problem, on the visible and invisible singularities respectively. The artefacts caused by the missing wedge can be explicitly characterised \cite{Frikel2013a,Katsevich1997} and examples of such streak artefacts and blurred boundaries can be seen in Figures \ref{bad recons} and \ref{bad TV}.
	
	Whilst the techniques developed here can apply to any limited angle tomography problem, we focus on the application of TEM for specific examples.}
	
	\begin{figure}
		\includegraphics[width=.45\linewidth]{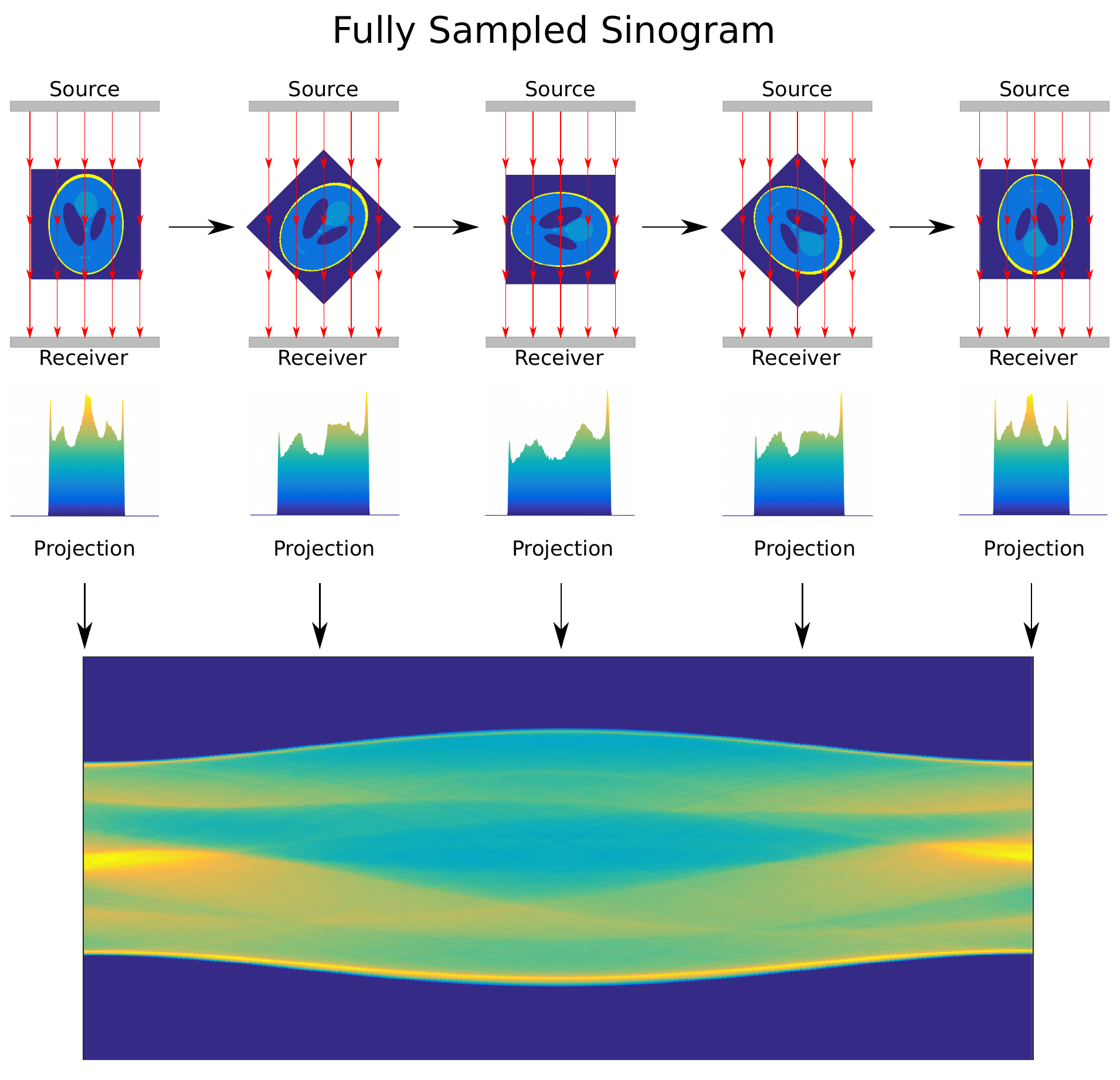}\hfill
		\includegraphics[width=.45\linewidth]{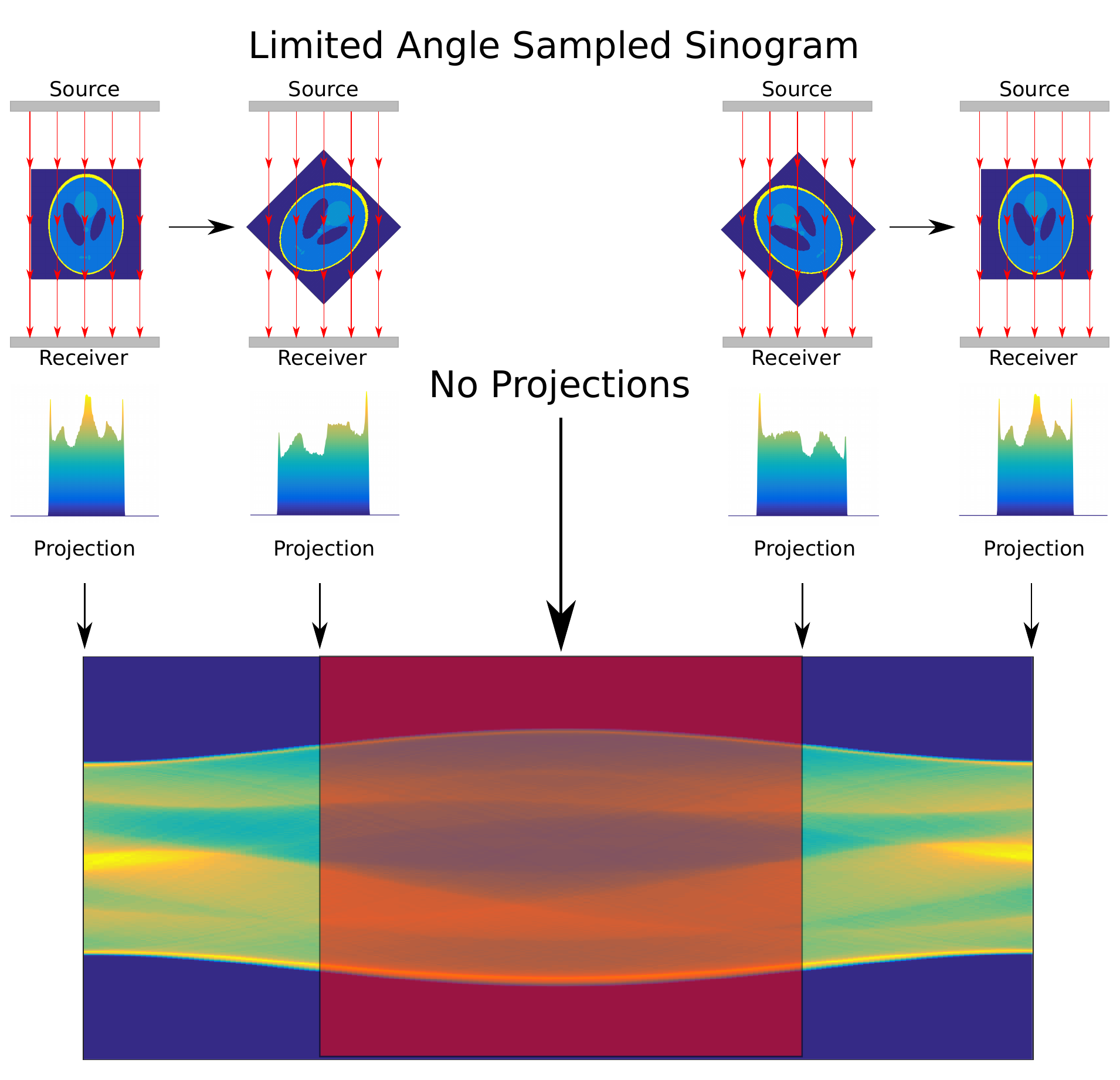}
		\caption{Diagrammatic representation of the acquisition of \newtext{2D} X-ray transform data \oldtext{in Electron Microscopy}\newtext{, the sinogram, }in both full range and limited angle acquisition. Note that measurement at $180^\circ$ is exactly a reflection of that at $0^\circ$. This symmetry allows us to consider a $180^\circ$ range of the sinogram as a full sample. In the limited angle setting we can only rotate the sample a small amount clockwise and anti-clockwise which results in missing data in the middle of the sinogram.}
		\label{limitedangle_eg}
	\end{figure}
	
	\subsection{Context and Proposed Model}
	\newtext{Traditional methods for X-ray Tomography reconstruction find approximate solutions to $S\C Ru = b$, constraining $\C Ru=v$ and only using prior knowledge of $u$ or the sinogram, $v$. There are three main methods which fit into this category:
	\begin{itemize}
		\item Filtered back projection (FBP) is a linear inversion method with smoothing on the sinogram, $v$, to account for noise \cite{Feldkamp1984,Flohr2003,Tang2006}.
		\item The Simultaneous Iterative Reconstruction Technique (SIRT) can be thought of as a preconditioned gradient descent on the function $\norm{S\C Ru-b}_2^2$ \cite{Gilbert1972,Kubel2005,Agulleiro2010,Spitzbarth2015}. Regularisation is then typically implemented by an early-stopping technique.
		\item Variational methods where prior knowledge is encoded in regularisation functionals have now been applied in this field for nearly a decade. In particular, the current state of the art in Electron Tomography is Total Variation (TV) regularisation \cite{Goris2012,Leary2013,Chen2013} where $u$ is encouraged to have a piecewise constant intensity. This will be introduced formally in Section \ref{prelims}.
	\end{itemize}

	FBP and SIRT are commonly used for their speed although variational methods like TV have quickly gained popularity as they enable prior physical knowledge to be explicitly incorporated in reconstructions. This added prior knowledge tends to stabilise the reconstruction process and Figure \ref{bad recons} gives examples where TV can vastly outperform FBP and SIRT when either the noise level is large or the angular range small. However, Figure \ref{bad TV} further shows the limitations of TV when high noise and limited angles are combined. The only difference between the Shepp-Logan phantom data shown in Figures \ref{bad recons} and \ref{bad TV} is that the former is clean data, in the image of the forward operator, whilst the latter has Gaussian white noise added. We see that as soon as there is a combined denoising/inpainting reconstruction problem, the TV prior on $u$ becomes insufficient to recover the structure of the sample.
	
	Recently, these traditional methods have received a revival through machine learning methods, see for instance \cite{Gu2017,Hammernik2017}. In both of these examples the main artefact reduction is a learned denoising step which only enforces prior knowledge on $u$.}

	\begin{figure}
		\includegraphics[width=\linewidth,trim={160pt 80pt 90pt 10pt},clip]{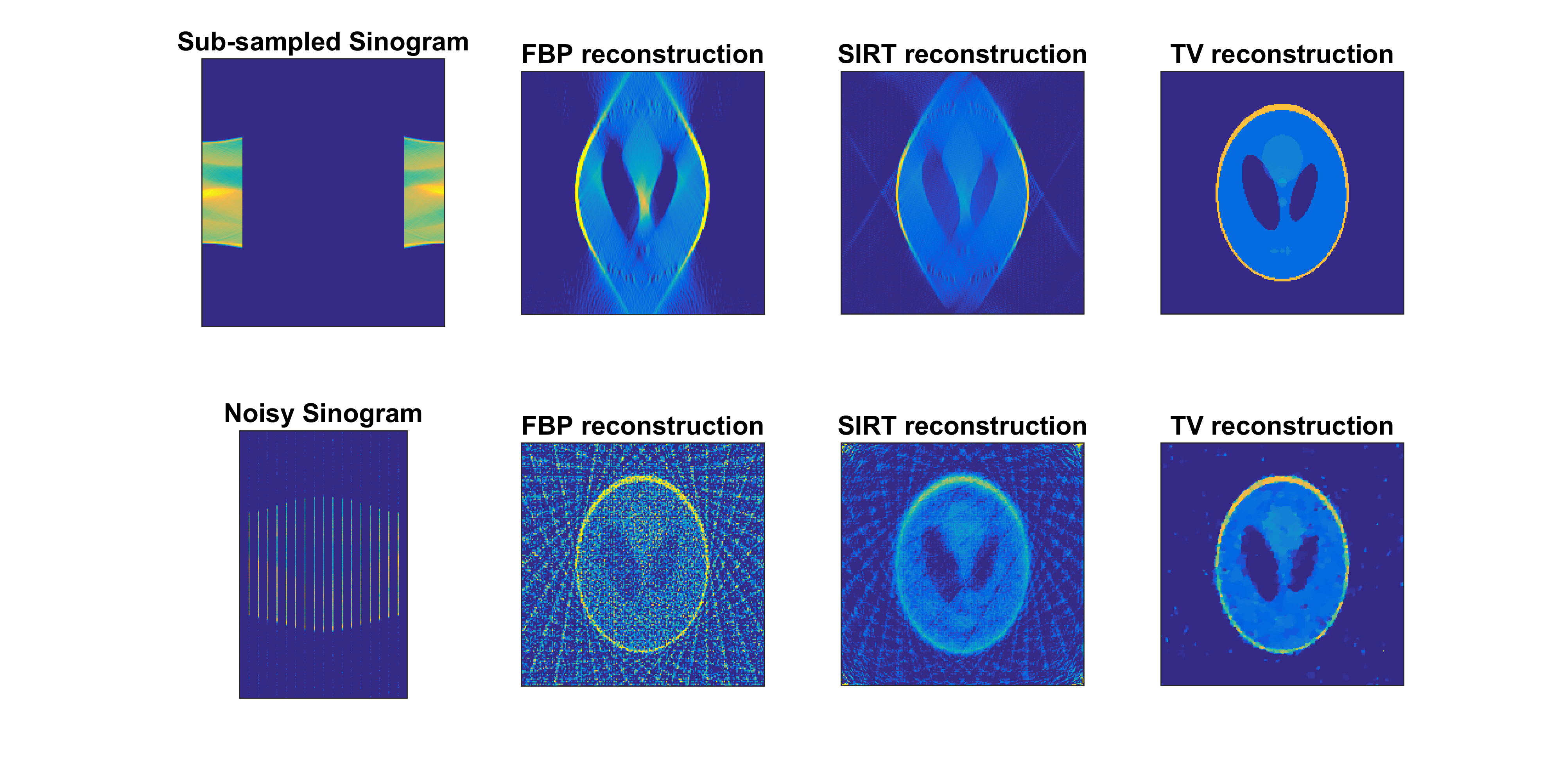}
		\caption{Demonstration of TV reconstruction in comparison to FBP and SIRT. The top row shows reconstructions from noise-less limited angle data and the bottom shows reconstructions from noisy limited view data (far left images). Comparing the columns, we immediately see that FBP and SIRT are much more prone to angular artefacts than TV. In both cases we notice that the TV reconstructions better show the broad structure of the phantom.}
		\label{bad recons}
	\end{figure}
	\begin{figure}
		\centering \includegraphics[width=.49\linewidth,trim={200 80 140 30},clip]{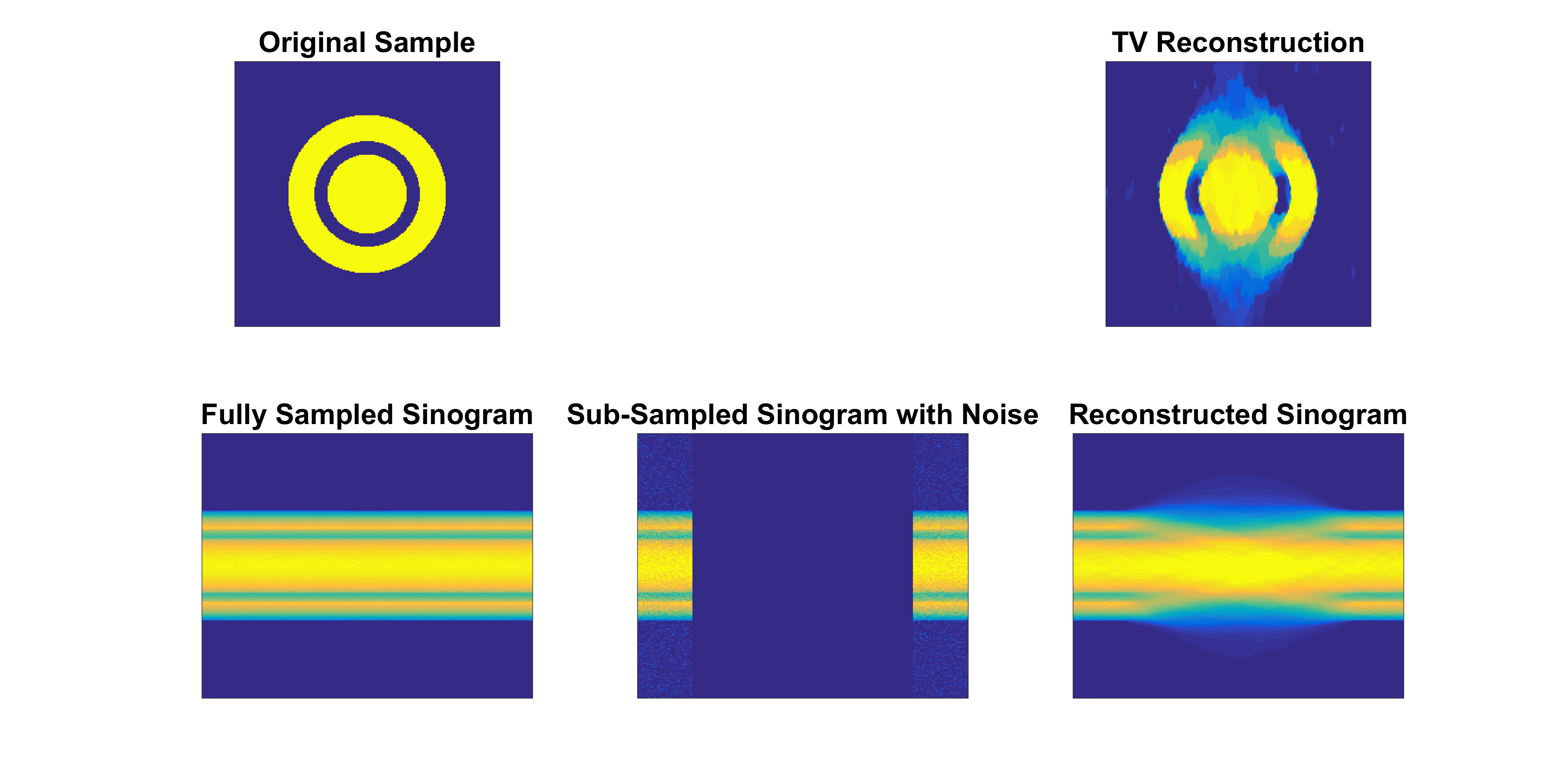}
		\includegraphics[width=.49\linewidth,trim={200 80 140 30},clip]{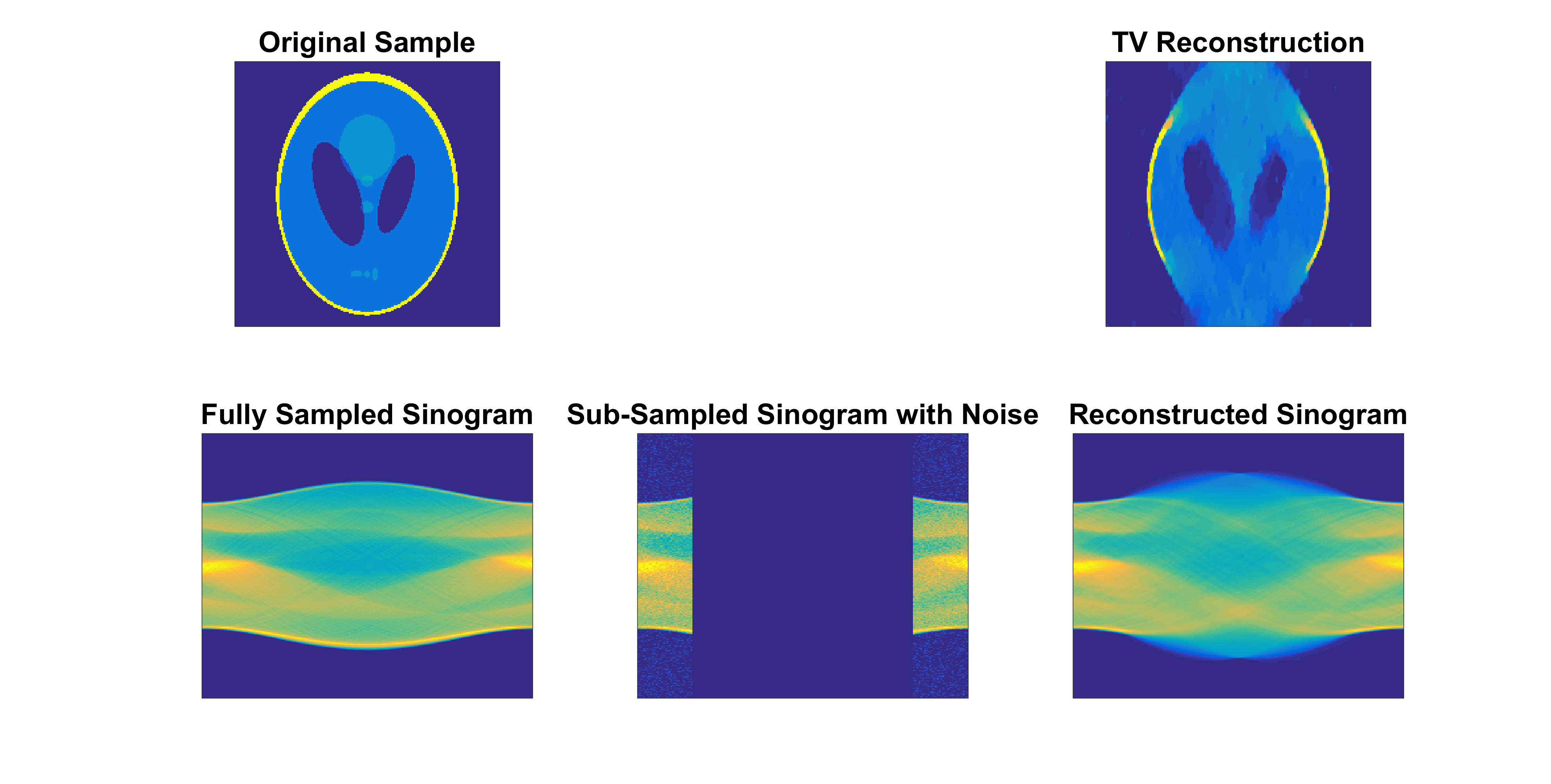}
		\caption{Examples when TV reconstructions cannot recover the global structures of samples. When there is a large missing wedge ($\sfrac23$ of data unseen) and noise on the projections then reconstructions exhibit characteristic blurring in the vertical direction. This can also be seen in the extrapolated region of the sinograms as a loss of structure.}
		\label{bad TV}
	\end{figure}
	
	The most common method that has been used to reconstruct pairs $(u,v)$ is to solve each inverse problem sequentially. Typically, we can express the pipeline for such methods as:
	\begin{align*}
	v&= \text{ optimal inpainted sinogram given $b$} 
	\\ u&= \text{ optimal reconstruction given $v$} 
	\end{align*}
	This has seen much success in heavy metal artefact reduction \cite{Kostler2006,Zhang2011} where a regularisation functional for the inpainting problem may be constructed from dictionary learning \cite{Li2014}, fractional order TV \cite{Zhang2011}, and Euler's Elastica \cite{Gu2006}. Euler's Elastica has also been used in the limited angle problem \cite{Zhang2017} along with more customised interpolation methods \cite{Kalke2014}. These approaches allow us to use prior knowledge on the sinogram to calculate $v$ and then spatial prior knowledge to calculate $u$ from $v$; at no point is the choice of $v$ influenced by our spatial prior. \newtext{A full joint approach allows us to go beyond this and use all of our prior knowledge to inform the choice of both $u$ and $v$. If our prior knowledge is consistent with the true data then this extra utilisation of our prior must have the potential to improve the recovery of both $u$ and $v$. To build a model for this framework we shall encode our } \newtext{In this paper, therefore, we propose a full joint approach which allows us to use all of our prior knowledge at once. To realise this idea we encode } prior knowledge and consistency terms into a single energy functional such that an optimal pair of reconstructions will minimise this joint functional, which we shall write as:
	\begin{equation}\label{basic model}
	 E(u,v) = \alpha_1d_1(\C Ru,v) + \alpha_2d_2(S\C Ru,b) + \alpha_3d_3(Sv,b) + \alpha_4r_1(u) +\alpha_5r_2(v)
	\end{equation}
	where $\alpha_i\geq0$ are weighting parameters, $d_i$ are appropriate distance functionals and $r_i$ are regularisation functionals which encode our prior knowledge. Note that choice of $d_2$ and $d_3$ are dictated by the data noise model. In what follows, $r_1$ is chosen to be the total variation.

	\newtext{Our choice for $r_2$, the sinogram regularisation, is based on theoretical and heuristic observations. Thirion \cite{Thirion1991} has shown that discontinuities in $u$ correspond to sharp edges in the sinogram. In Figure \ref{bad TV} we also see that blurred reconstructions correspond to loss of structure in the sinogram. Therefore, $r_2$ will be chosen to detect sharp features in the given data and preserve these through the inpainting region. The exact form of $r_2$ will be formalised in Section \ref{The Joint Model}.
	
	A typical advantage of joint models is that they generalise previous ones. For instance, if we let $\alpha_2,\alpha_4\to\infty$ then we recover the TV reconstruction model. Alternatively, if we let $\alpha_3,\alpha_5\to\infty$ then we recover a method which performs the inpainting and then the reconstruction sequentially, as in \cite{Zhang2011,Gu2006,Zhang2017}. Recent work \cite{Burger2014} has shown that such a joint approach can be advantageous in similar situations but closest to our approach is that of \cite{Dong2013} where $r_1$ and $r_2$ were chosen to encode wavelet sparsity in both $u$ and $v$. 
	We shall demonstrate that the flexibility of our joint model, \eqref{basic model}, can allow for a better adaptive fitting to the data.}

	\begin{figure}
		\centering \includegraphics[width=.25\linewidth,trim={140 400 770 20},clip]{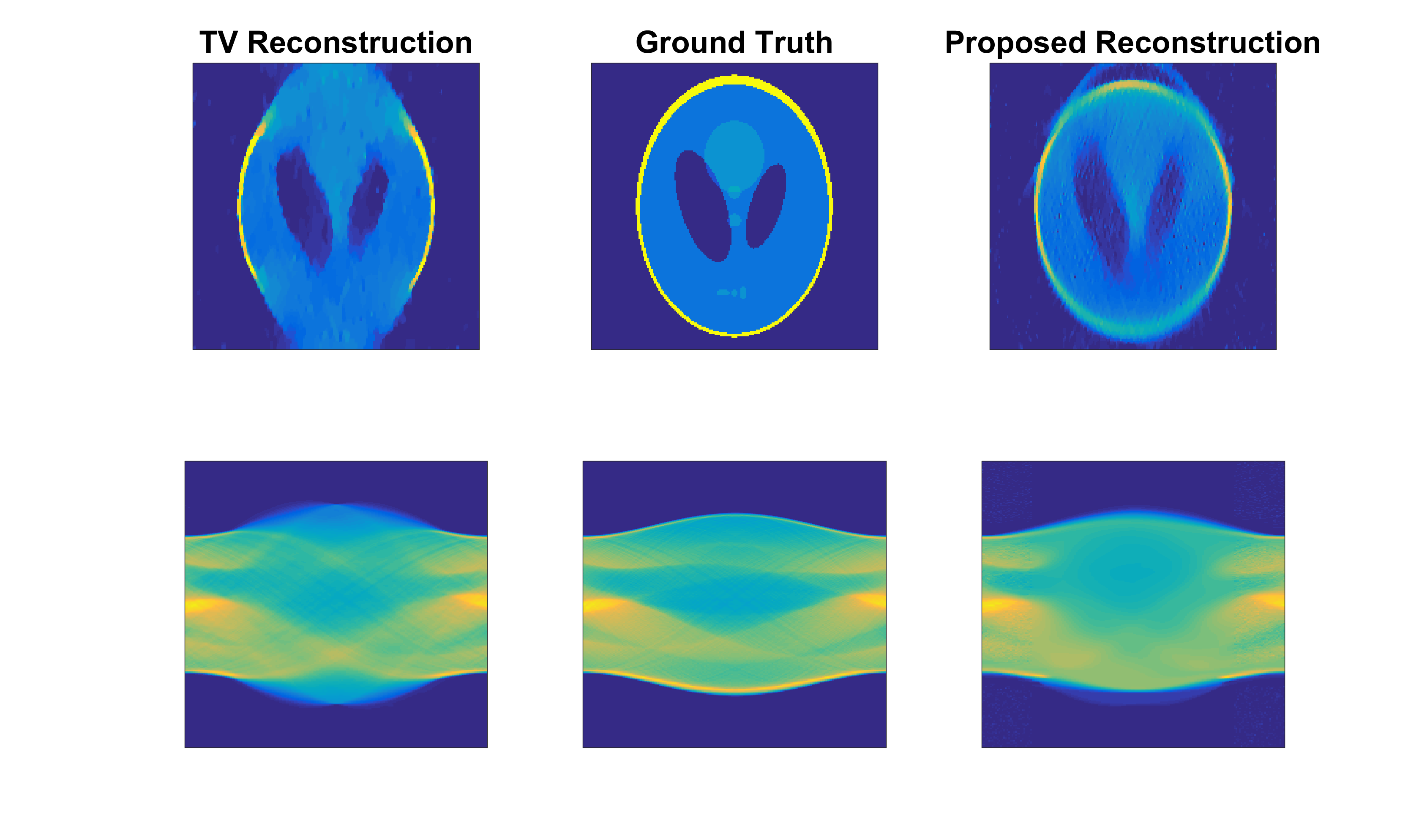}
		\includegraphics[width=.31\linewidth,trim={780 400 55 20},clip]{phantom_res}
		\\\centering \includegraphics[width=.25\linewidth,trim={140 80 770 420},clip]{phantom_res}
		\includegraphics[width=.31\linewidth,trim={780 80 55 420},clip]{phantom_res}
		\caption{Demonstration of the improvement which can be achieved by using a model as in \eqref{basic model}. \oldtext{Top } \newtext{Left hand } images show state of the art reconstructions using Total Variational regularisation ($\alpha_1=\alpha_3=\alpha_5=0$). This reconstruction clearly shows the characteristic blurring artefacts at the top and bottom. In our proposed joint reconstruction (\oldtext{bottom line} \newtext{right hand}) we can minimise these effects.}
		\label{TV vs joint}
	\end{figure}

	\subsection{Overview and Contributions}
	The main contribution of this work is to provide a framework for building models of the form described in \eqref{basic model} and provide new proofs for	a numerical scheme for minimising these functionals \oldtext{, even when these problems become non-smooth and non-convex}. \newtext{This numerical scheme is valid for a very large class of non-smooth and non-convex functionals $r_i$ and thus could be used in many other applications.}
	
	Section \ref{prelims} first outlines the necessary concepts and notation needed to formalise the statement of our specific joint model in Section \ref{The Joint Model}. It will become apparent that the main numerical requirements of this work will require minimising a functional which is neither convex or smooth. Section \ref{optimisation} will start by reviewing recent work from Ochs et. al. \cite{Ochs2017} and we then provide alternative concise and self-contained proofs. Our main contribution here is to extend the existing results to an alternating (block) descent scenario. Finally, we present numerical results including two synthetic phantoms and experimental Electron Microscopy data where the limited angle situation occurs naturally.
	
	\section{Preliminaries}\label{prelims}
	\subsection{The X-ray Transform}
	The principal forward model for X-ray tomography is provided by the X-ray transform which can be defined for any bounded domain, $\Omega\subset\F R^n$, by
	\begin{equation}\label{radon}
	\C R\colon L^1(\Omega,\F R)\to L^1(\C S^{n-1}\times \F R^n,\F R) \text{ such that } \C Ru(\theta,y) = \int_{x = y + t\theta, t\in \F R} u(x)dt
	\end{equation}
	where $\C S^{n-1}=\{s\in\F R^n\st |s| = 1\}$. \oldtext{In most applications we have $n=3$ although in many cases (e.g. EM, X-ray CT) the scanners proceed slice-wise and so we have 
	$ \C R\colon L^1(\Omega,\F R)\to L^1(\C S^{1}\times \F R^3,\F R) \text{ such that } \C Ru(\theta,y) = \int_{t\in\F R} u(y_1+t\theta_1,y_2+t\theta_2,y_3)dt$
	Consequently, we shall provide our numerical examples in 2D and this can be extended to 3D by a trivial extension of our proposed model.} \newtext{In this work, for simplicity, we will only be using $n=2$ although the case $n=3$ is completely analogous.}
	
	\subsection{Total Variation Regularisation}
	Total Variation (TV) regularisation is extremely common across many fields of image processing \cite{Beck2009,Chan2006,Benning2017}. The definition of the (isotropic) TV semi-norm on domains $\Omega\subset\F R^n$ is stated as
	\begin{multline}\label{TV}
	g\in L^1(\Omega,\F R) \implies \TV(g) = \sup\left\{ \int \IP{g(x)}{\op{div}(\varphi)(x)}dx\right. \\\left.\vphantom{\int}\st \varphi\in C_c^\infty(\Omega,\F R^n), |\varphi(x)|_2\leq 1\text{ for all } x\right\}
	\end{multline}
	The intuition behind this is that when $g\in W^{1,1}(\Omega,\F R)$ then we have exactly \[\TV(g) = \norm{\nabla g}_{2,1} = \int |\nabla g(x)|_2dx\]
	From this point forward we shall write the $\norm{\cdot}_{2,1}$ representation where $|\nabla g|$ is to be understood as a measure if necessary.
	We shall denote the space of Bounded Variation as \[\op{BV}(\Omega, \F R) = \{ g\in L^1\st \TV(g) < \infty\} \]
	One property that we shall use about the space of Bounded Variation is $\op{BV}\subset L^2\subset L^1$ which holds whenever $\Omega$ is compact by the Poincar\'e inequality: $\norm{g-\sfrac{\int g}{\int 1}}_2 \lesssim \op{TV}(g)$.
	
	The most common way to enforce this prior in reconstruction or inpainting problems is generalised Tikhonov regularisation, which gives us the basic reconstruction method \cite{Goris2012,Leary2013}.
	\begin{equation}
	u = \argmin_{u\geq 0} \frac12\norm{S\C Ru-b}_2^2 + \lambda\TV(u) \text{ for some } \lambda\geq0
	\end{equation}
	The parameter $\lambda$ is a regularisation parameter, which allows to enforce more or less regularity, depending on the quality of the data $b$.

	\subsection{Directional Total Variation Regularisation}
	For our sinogram regularisation functional we shall use a directionally weighted TV penalty, motivated by the TV diffusion model developed by Joachim Weickert \cite{Weickert1998} for various imaging techniques including denoising, inpainting and compression. Such an approach has already shown great ability for enhancing edges in noisy or blurred images, and preserves line structures across inpainting regions \cite{Berkels2006,Estellers2015,Bertalmio2000}. The heuristic for our regularisation on the sinogram was described in Figure \ref{bad TV} and we shall encode it in an anisotropic TV penalty which shall be written as 
	\[\DTV(v) = \int |A(x)\nabla v(x)| dx = \norm{A\nabla v}_{2,1} \text{ for some anisotropic } A\colon\F R^2\to \F R^{2\times 2}.\]
	The power of such a weighted extension of TV is that once a line is detected, either known beforehand or detected adaptively, we can embed this in $A$ and enhance or sharpen that line structure in the final result. The general form that we choose for $A$ is 
	\begin{equation}\label{matrix form}
	\begin{split}
	A(x) = &\;c_1(x)\vec e_1(x)\vec e_1(x)^T + c_2(x)\vec e_2(x)\vec e_2(x)^T \\&\text{ such that } \vec e_i\colon \F R^2\to \F R^2, |\vec e_i(x)| = 1, \IP{\vec e_1(x)}{\vec e_2(x)} = 0
	\end{split}
	\end{equation}
	i.e. \[\DTV(v) = \int \sqrt{c_1^2|\IP{\vec e_1}{\nabla v}|^2 + c_2^2|\IP{\vec e_2}{\nabla v}|^2} dx.\]
	Examples of this are presented in Figure \ref{DTV inpainting}. Note that the choice $c_1=c_2$ recovers the traditional TV regulariser but for $|c_1|\ll c_2$ we allow for much larger (sparse) gradients in the direction of $\vec e_1$. This allows for large jumps in the direction of $\vec e_1$ whilst maintaining flatness in the direction of $\vec e_2$. In order to generate these parameters we follow the construction of Weickert \cite{Weickert1998}. Given a noisy image, $d$, we can construct the structure tensor:
	\[(\nabla d_\rho \nabla d_\rho^T)_\sigma(x) = \lambda_1(x)\vec e_1(x)\vec e_1(x)^T + \lambda_2(x)\vec e_2(x)\vec e_2(x)^T \text{  such that } \lambda_1(x)\geq \lambda_2(x)\geq0\]
	where \[d_\rho(x) = [d\star \exp(-\sfrac{|\cdot|^2}{2\rho^2})] (x)= \int \exp(-\sfrac{|y-x|^2}{2\rho^2})d(y) \text{ etc.}\]
	denotes convolution with the heat kernel. This eigenvalue decomposition is typically very informative in constructing $A$. If we define
	\[ \Delta(x) = \lambda_1(x)-\lambda_2(x) \text{ is coherence} \qquad \Sigma(x) = \lambda_1(x)+\lambda_2(x) \text{ is energy}\]
	then the eigenvectors give the alignment of edges and $\Delta, \Sigma$ characterise the local behaviour, as in Figure \ref{energy and coherence}. In particular, we simplify the model to
	\begin{equation}\label{direction tensor}
	A_d(x) \coloneqq c_1(x|\Delta,\Sigma)\vec e_1(x)\vec e_1(x)^T + c_2(x|\Delta,\Sigma)\vec e_2(x)\vec e_2(x)^T
	\end{equation}
	where the only parameters left to choose are $c_i$. Typical examples of include 
	\[c_1 = \frac{1}{\sqrt{1+\Sigma^2}},\qquad c_2 = 1\]
	\[c_1 = \epsilon,\qquad c_2 = \epsilon+\exp\left(-\sfrac{1}{\Delta^2}\right) \text{ for some } \epsilon>0\]
	The key idea here is that $c_1\ll c_2$ near edges and $c_1=c_2$ on flat regions. In practice $d$ will also be an optimisation parameter and so we shall require a regularity result on our choice of $d\mapsto A_d$, now characterised uniquely by our choice of $c_i$.
	\begin{figure}
		\centering
		\includegraphics[width=.45\linewidth,trim={50 75 50 20},clip]{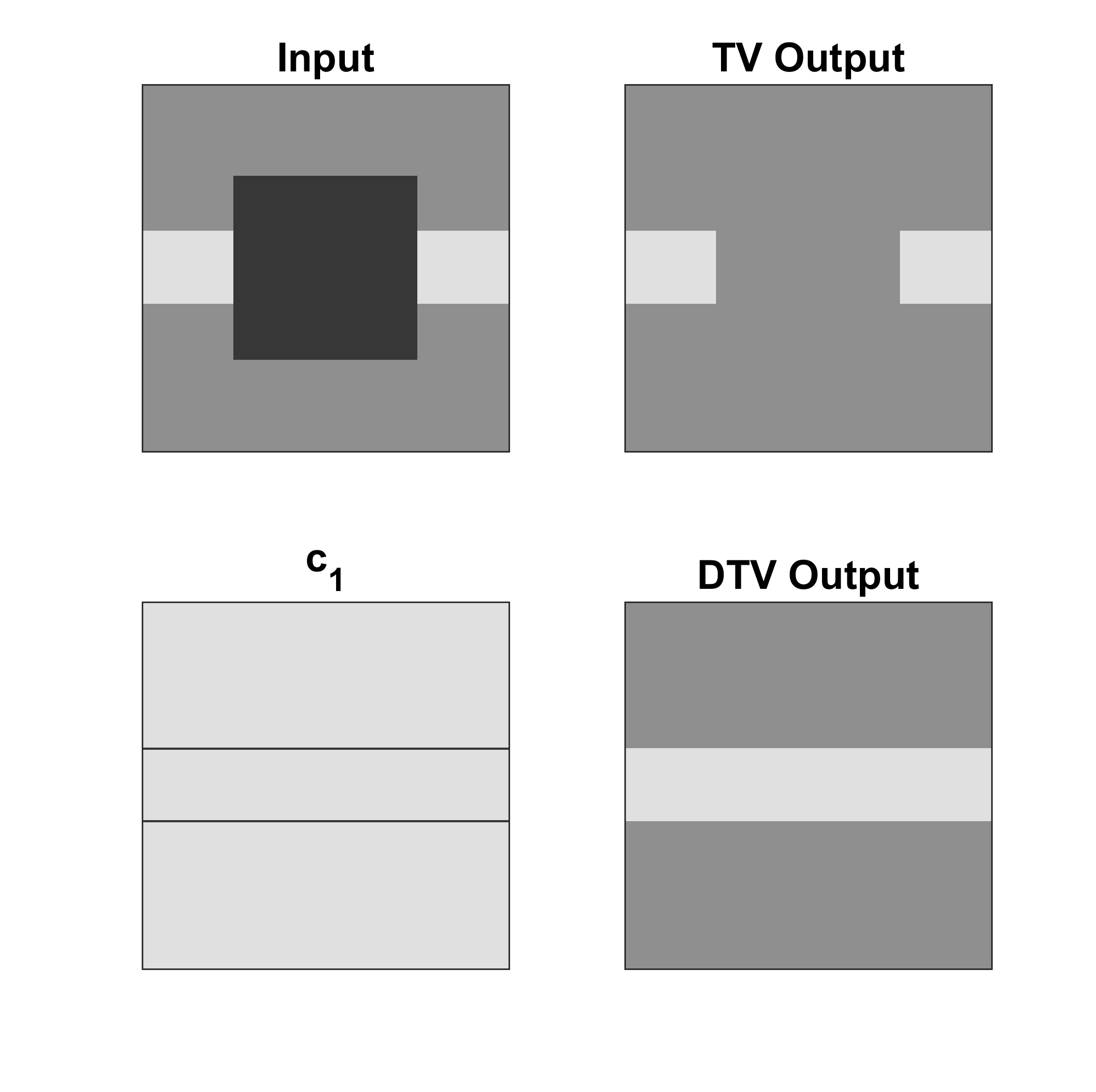}\hfill
		\includegraphics[width=.45\linewidth,trim={50 75 50 20},clip]{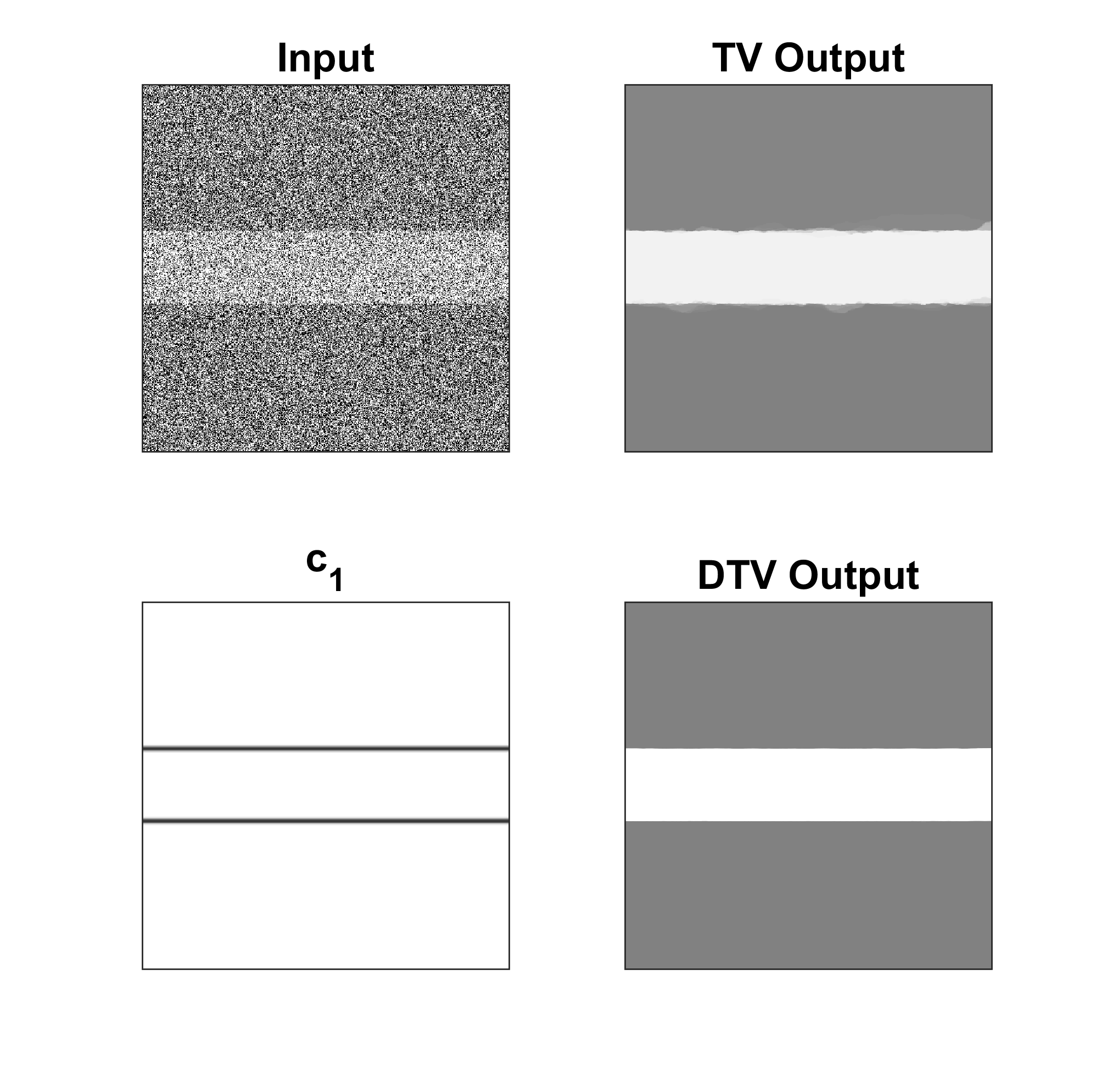}
		
		\caption{Examples comparing TV with directional TV for inpainting and denoising. Both examples have the same edge structure and so parameters in \eqref{matrix form} are the same in both. DTV uses $c_2=1$ and $c_1$ as the indicator (0 or 1) shown in the bottom left plot, TV is the case $c_1=c_2=1$. Left block: Top left image is inpainting input where the dark square shows the inpainting region. The structure of $c_1$ allows DTV (bottom right) to connect lines over arbitrary distances, whereas TV inpainting (top right) will fail to connect the lines if the horizontal distance is greater than the vertical separation of the edges. Right block: Top left image is denoising input. DTV has two advantages. Firstly, the structure of $c_1$ recovers a much straighter line than that in the TV reconstruction. Secondly, TV penalises jumps equally in each direction and so the contrast is reduced, DTV is able to penalise noise oscillations independently from edge discontinuities which allows us to maintain much higher contrast.}\label{DTV inpainting}
	\end{figure}
	\begin{figure}
		\centering\includegraphics[width=.7\textwidth,trim={0 80 0 70},clip]{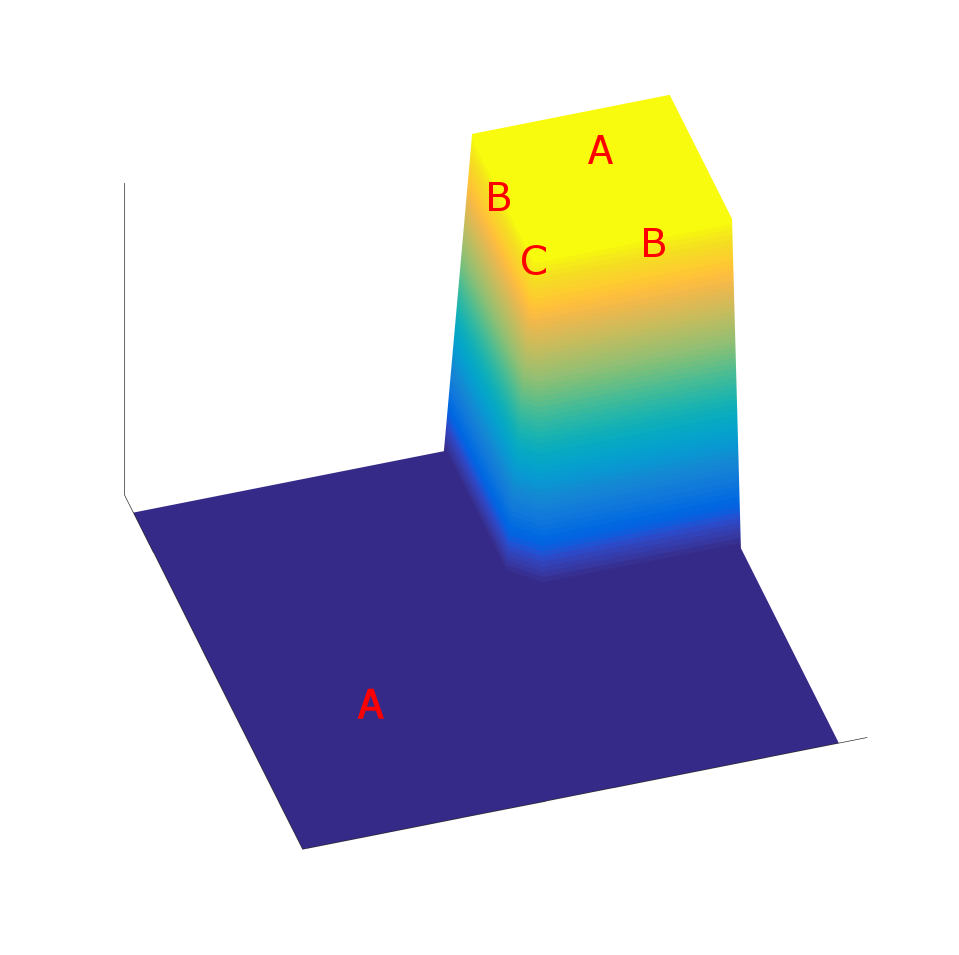}
		\caption{Surface representing a characteristic image, $d$, to demonstrate the behaviour of $\Sigma$ and $\Delta$. Away from edges (A) we have $\Sigma\approx\Delta\approx 0$. On simple edges (B) we have $\Sigma\approx \Delta \gg  0$ and, finally, at corners (C) we have $\Sigma\gg \Delta$.}\label{energy and coherence}
	\end{figure}
	\begin{theorem}\label{smooth tensor}
		If
		\begin{enumerate}
			\item $c_i$ are $2k$ times continuously differentiable in $\Delta$ and $\Sigma$, $k\geq1$
			\item $c_1(x|0,\Sigma) = c_2(x|0,\Sigma)$ for all $x$ and $\Sigma\geq0$
			\item $\partial_\Delta^{2j-1} c_1(x|0,\Sigma) = \partial_\Delta^{2j-1} c_2(x|0,\Sigma) = 0$ for all $x$ and $\Sigma\geq0, j=1\ldots,k$
		\end{enumerate}
		Then $A_d$ is $C^{2k-1}$ with respect to $d$ for all $\rho>0,\sigma\geq0$.
	\end{theorem}
	\begin{remark}\hfill
		\begin{itemize}
			\item Property (ii) is necessary for $A_d$ to be well defined and continuous for all fixed $d$
			\item If we can write $c_i = c_i(\Delta^2,\Sigma)$ then property (iii) holds trivially
		\end{itemize}
	\end{remark}
	\noindent The proof of this theorem is contained in \ref{App: smooth tensor}.
	
	\section{The Joint Model}\label{The Joint Model}
	Now that all of the notation and concepts have been defined we can formalise the statement of our particular joint model of the form in \eqref{basic model}:
	\begin{itemize}
		\item The forward operator, $\C R$, is the X-ray transform from \eqref{radon}
		\item The desired reconstructed sample, $u\in \op{BV}(\Omega,\F R)$ on some domain $\Omega$
		\item The noisy sub-sampled data, $b\in L^1(\Omega',\F R)$ on some $\Omega'\Subset \C S^1\times\F R_{\geq0}$. We extend such that $b|_{\Omega'^c}=0$ for notational convenience.
		\item The full reconstructed sinogram, $v\in L^1(\C S^1\times\F R_{\geq0},\F R)$
	\end{itemize}
	We also define $S = S_{\Omega'}$ to be the indicator function on $\Omega'$. The joint model that we thus propose is 
	\begin{equation}\label{model}\begin{split}
	(u, v) = \argmin_{u\geq 0} E(u,v) = & \;\newtext{\argmin_{u\geq0}}\frac{1}{2}\norm{\C Ru-v}^2_{\alpha_1} + \frac{\alpha_2}{2}\norm{S\C Ru-b}_2^2
	\\ &\qquad\quad+ \frac{\alpha_3}{2}\norm{Sv-b}_2^2 + \beta_1\op{TV}(u) + \beta_2\DTV_{\C Ru}(v)
	\end{split}\end{equation}
	where \newtext{\[\DTV_{\C Ru}(v) = \norm{A_{\C Ru}\nabla v}_{2,1}\]}
	and $\alpha_i,\beta_i>0$ are weighting parameters\newtext{, $A_{Ru}$ as defined in \eqref{direction tensor}}.	$\alpha_1$ is embedded in the norm as it is a spatially varying weight, taking different values inside and outside of $\Omega'$. We note that the norms involving $b$ are determined by the noise model of the acquisition process, in this case Gaussian noise. The final metric pairing $\C Ru$ and $v$ was free to be chosen to promote any structural similarity between the two quantities. We have chosen the squared $L^2$ norm for simplicity although if some structure is known to be important then there is a wide choice of specialised functions from which to choose (e.g. \cite{Ehrhardt2015}). 
	
	\newtext{The choice of regularisation functionals reflects prior assumptions on the expected type of sample; all of the examples shown later will follow these assumptions. The isotropic TV penalty is chosen as $u$ is expected to be piece-wise constant. This will reduce oscillations from $u$ and favour `stair-case' like changes of intensity over smooth ones. The assumptions of our regularisation on $v$ must also be derived from expected properties of $u$. What is known from \cite{Thirion1991}, and can be seen in Figure \ref{bad TV}, is that discontinuities of $u$ along curves in the spatial domain, say $\gamma$, generate a so called `dual curve' in the sinogram. $\C Ru$ will also have an edge, although possibly not a discontinuity, along this dual curve. Thus, perpendicular to the dual curve $v$ should have sharp features and parallel to the dual curve intensity should vary slowly. The assumption of our regularisation is that if a dual curve is present in the visible component of the data then it should correspond to some $\gamma$ in the spatial domain. The extrapolation of this dual curve must behave like the boundary of a level-set of $u$, i.e. preserve the sharp edge and slow varying intensities in $v$. The main influence of this regularisation is in the inpainting region and so any artefacts it introduces should also only effect edges corresponding to these invisible singularities, including streaking artefacts. Another bias that is present is an assumption that dual curves are themselves smooth. In the inpainting region, this will encourage dual curves with low curvature thus invisible singularities are likely to follow near-circular arcs in the spatial domain. Final parameter choices, such as $\alpha_i, \beta_i$ and $c_i$, are not necessary at this point and will be chosen in Section \ref{numerical details}.}

	The immediate question to ask is whether this model is well posed. For a non-convex function we typically cannot expect numerically to find global minimisers but the following result shows we can expect some convergence to local minima. We present the following result which justifies looking for minima of this functional.
	\begin{theorem}\label{conv thm}
		If 
		\begin{itemize}
			\item $c_i$ are bounded away from 0
			\item $\rho>0$
			\item $A_d$ is differentiable in $d$
		\end{itemize}
		then sub-level sets of $E$ are weakly compact in $L^2(\Omega,\F R)\times L^2(\F R^2,\F R)$ and $E$ is weakly lower semi-continuous. i.e. for all $(u_n,v_n)\in L^2(\Omega,\F R)\times L^2(\F R^2,\F R)$:
		\[E(u_n,v_n) \text{ uniformly bounded implies a subsequence converges weakly} \]
		\[\liminf E(u_n,v_n) \geq E(u,v) \text{ whenever } u_n\rightharpoonup u, v_n\rightharpoonup v\]
	\end{theorem}
	\noindent The proof of this theorem is contained in \ref{App: conv thm}. This theorem justifies numerical minimisation of $E$ because it tells us that all \oldtext{(locally)} descending sequences ($E(u_n,v_n)\leq E(u_{n-1},v_{n-1})$) have a convergent subsequence and any limit point must minimise $E$ over the original sequence.

	\section{Numerical Solution}\label{optimisation}
	To address the issue of convergence we shall first generalise our functional and prove the result in the general setting. Functionals will be of the form $E\colon X\times Y\to \F R$ where $X$ and $Y$ are Banach spaces and $E$ accepts the decomposition \[E(x,y) = f(x,y) + g(J(x,y))\]
	such that:
	\begin{align}
	&\text{Sub-level sets of } E \text{ are weakly compact}\label{myFunc-start}
	\\&f\colon X\times Y\to \F R \text{ is jointly convex in } (x,y) \text{ and bounded below} 
	\\&g\colon Z\to \F R \text{ is a semi-norm on Banach space } Z \text{, i.e. for all } t\in\F R, z,z_1,z_2\in Z \notag
	\\&\qquad g(z)\geq0, g(tz) = |t|g(z)\text{ and } g(z_1+z_2)\leq g(z_1)+g(z_2)
	\\&J\colon X\times Y\to Z \text{ is } C^1 \text{ and for all } \oldtext{ \text{local } x,y } \newtext{K\Subset X\times Y}, \;\exists L_x,L_y<\infty \newtext{\text{ such that } \forall (x,y)\in K}\notag
	\\&\qquad g(J(x+dx,y)-J(x,y) -\nabla_xJ(x,y)dx) \leq L_x\norm{dx}_X^2
	\\&\qquad g(J(x,y+dy)-J(x,y) -\nabla_yJ(x,y)dy) \leq L_y\norm{dy}_Y^2\label{myFunc-end}
	\end{align}
	Note if $g$ is a norm then $L_x$ can be chosen to be the standard Lipschitz factor of $\nabla_xJ$. If $J$ is twice \newtext{Fr\`echet-}differentiable then these constants must be finite. In our case:
	\begin{align*}
	f(x,y) &= \frac12\norm{Rx-y}_{\alpha_1}^2+\frac{\alpha_2}{2}\norm{S\C Rx-b}^2
	\\&\qquad+\frac{\alpha_3}{2}\norm{Sy-b}^2 + \beta_1\TV(x) + \splitln{0}{x\geq 0}{\infty}{\text{else}}
	\\g(z) &= \beta_2\norm{z}_{2,1}
	\\J(x,y) &= A_{Rx}\nabla y\implies \tau_x \sim \beta_2\norm{\nabla^2A_{\cdot}}\norm{R}\TV(y), \tau_y = 0
	\end{align*}
	Finiteness of $\norm{\nabla^2 A}$ and weak compactness of sub-level sets are given by Theorems \ref{smooth tensor} and \ref{conv thm} respectively. The alternating descent algorithm is stated in Algorithm \ref{alg}.
	\begin{algorithm}
		\caption{}
		\begin{algorithmic}
			\REQUIRE Any $x_0\in X$, $\tau_x, \tau_y\geq 0$.
			\STATE $n\leftarrow 0$
			\WHILE{not converged}
			\STATE $n\leftarrow n+1$
			\vspace{-30pt}
			\STATE{\begin{multline}
			x_{n} = \argmin_{x\in X} f(x,y_{n-1}) + \tau_x\norm{x-x_{n-1}}_{X}^2 
			\\+g(J(x_{n-1},y_{n-1})+\nabla_xJ(x_{n-1},y_{n-1})(x-x_{n-1})) \label{xstep}
			\end{multline}}
			\vspace{-45pt}
			\STATE{\begin{multline}
				y_{n} = \argmin_{y\in Y} f(x_n,y) + \tau_y\norm{y-y_{n-1}}_{Y}^2
				\\\qquad\qquad +g(J(x_{n},y_{n-1})+\nabla_yJ(x_{n},y_{n-1})(y-y_{n-1}))\label{ystep}
				\end{multline}}
			\vspace{-25pt}
			\ENDWHILE
			\ENSURE $(x_n,y_n)$
		\end{algorithmic}
		\label{alg}
	\end{algorithm}
	Classical alternating proximal descent would take $x_{n+1} = \argmin E(x,y_n)+\tau_x\norm{x-x_n}_2^2$ but because of the complexity of $A_{\C R u}$ each sub-problem would have the same complexity as the full problem, making it computationally infeasible. \oldtext{The classical solution to this would have been to find a smooth approximation of $\norm\cdot_{2,1}$ with a small maximum error, $\epsilon$. This now decomposes $E(x,y_n)$ into the sum of a differentiable function and a convex function for which there exist many minimisation algorithms, e.g. \cite{Beck2009}. There are two problems with this approach, the original energy functional has been changed and the new energy is ill-conditioned with Lipschitz gradient $O(\sfrac{|\partial_dA|}\epsilon)$. }By linearising $A_d$ we overcome this problem as both sub-problems are convex, for which there are many standard solvers such as \cite{Chambolle2011, Mosek}. This second approach is similar to that of the ProxDescent algorithm \cite{Drusvyatskiy2016,Ochs2017}. We extend this algorithm to cover alternating descent and achieve equivalent convergence guarantees. Using Algorithm \ref{alg}, our statement of convergence is the following theorem.
	\begin{theorem}\label{alternating thm}
		\hfill\\Convergence of alternating minimisation: If $E$ satisfies \eqref{myFunc-start}-\eqref{myFunc-end} \newtext{and $(x_n,y_n)$ are a sequence generated by Algorithm \ref{alg}} then 
		\begin{itemize}
			\item $E(x_{n+1},y_{n+1}) \leq E(x_n,y_n)$ for each $n$.
			\item A subsequence of $(x_n,y_n)$ must converge weakly in $X\times Y$
			\item If $\{(x_n,y_n)\st n=1,\ldots\}$ is contained in a finite dimensional space then every limit point of $(x_n,y_n)$ must be a critical point (as will be defined in Definition \ref{critical point}) of $E$ in both the direction of $x$ and $y$.
		\end{itemize}
	\end{theorem}
	This result is the parallel of Lemma 10, Theorem 18 and Corollary 21 in \cite{Ochs2017} without the alternating or block descent setting. There is some overlap in the analysis for the two settings although we present an independent proof which is more direct and we feel gives more intuition for our more restricted class of functionals. The rest of this section is now dedicated to the proof of this theorem.
	
	For notational convenience we shall compress notation such that:
	\[f_{n,m} = f(x_n,y_m), \quad J_{n,m} = J(x_n,y_m), \quad E_{n,m}=E(x_n,y_m) \text{ etc.}\]	
	
	\subsection{Sketch Proof}
	The proof of Theorem \ref{alternating thm} will be a consequence of two lemmas.
	\begin{itemize}
		\item \oldtext{(Lemma \ref{global convergence})} \newtext{In Lemma \ref{global convergence} we show } for $\tau_x,\tau_y$ \newtext{(Algorithm \ref{alg})} sufficiently large, the sequence $E_{n,n}$ is monotonically decreasing and sequences $\norm{x_n-x_{n-1}}_X, \norm{y_n-y_{n-1}}_Y$ converge to 0. This follows by a relatively standard sufficient decrease argument as seen in \cite{iPalm,Ochs2017, Liang2016}.
		\item \oldtext{(Lemma \ref{critical convergence})} \newtext{In Lemma \ref{critical convergence}} we first define a critical point for functions which are non-convex and non-differentiable. If a subsequence of our iterates converges in norm then the limit must be a critical point in each of the two axes. Note that it is very difficult to expect more than this in such a general setting, for instance Example \ref{2-axis example} shows a uniformly convex energy which shows this to be sharp. The common technique for overcoming this is assuming differentiability in the terms including both $x$ and $y$ \cite{iPalm, iPiano, Palm}. \oldtext{The necessity of our algorithm is that the non-convex term is also non-differentiable and so it is unclear what would be required to make this statement stronger.} \newtext{These previous results and algorithms are not available to us as we allow non-convex terms which are also non-differentiable.}
	\end{itemize}
	\begin{example}\label{2-axis example}
		Define $E(x,y) = \max(x,y)+x^2+y^2$ for $x,y\in\F R$. This is clearly jointly convex in $(x,y)$ and thus a simple case of functions considered in Theorem \ref{alternating thm}. Suppose $(x_0,y_0) = (0,0)$ then
		\begin{align*}
		x_1 &= \argmin E(x,y_0) + \tau_x(x-x_0)^2 = 0
		\\y_1 &= \argmin E(x_1,y) + \tau_y(y-y_0)^2 = 0
		\end{align*}
		Hence $(0,0)$ is a limit point of the algorithm but it is not a critical point, $E$ is uniformly convex and so it has only one critical point, $(-\sfrac12,-\sfrac12)$.
	\end{example}
	
	\subsection{Sufficient Decrease Lemma}
	In the following we prove the monotone decrease property of our energy functional between iterations.
	\begin{lemma}\label{global convergence}
		If for each $n$ \[\tau_x \geq L_x+\tau_X, \qquad \tau_y \geq L_y+\tau_Y\]
		for some $\tau_X,\tau_Y\geq0$ then \[\sum^\infty\tau_X\norm{x_n-x_{n-1}}_{ X}^2 + \tau_Y\norm{y_n-y_{n-1}}_{ Y}^2 \leq E(x_0,y_0)\]
		and \[E(x_{n+1},y_{n+1})\leq E(x_n,y_n) \text{ for all } n\]
	\end{lemma}
	\begin{proof}
	Note by Equations \eqref{xstep}, \eqref{ystep} (definition of our sequence) we have
	\begin{align}
	f_{n+1,n} + g(J_{n,n}+\nabla_xJ_{n,n}(x_{n+1}-x_n)) + \tau_x\norm{x_{n+1}-x_n}_{ X}^2 &\leq E_{n,n} \label{gc1}
	\\f_{n+1,n+1} + g(J_{n+1,n}+\nabla_yJ_{n+1,n}(y_{n+1}-y_{n})) + \tau_y\norm{y_{n+1}-y_{n}}_{ Y}^2 &\leq E_{n+1,n} \label{gc3}
	\end{align}
	Further, by application of the triangle inequality for $g$ and the mean value theorem we have
	\begin{align}
	g(J_{n+1,n}) - &g(J_{n,n}+\nabla_xJ_{n,n}(x_{n+1}-x_{n}))+ \tau_X\norm{x_{n+1}-x_{n}}_{ X}^2\notag
	\\&\qquad \leq g(J_{n+1,n} -J_{n,n}-\nabla_xJ_{n,n}(x_{n+1}-x_{n}))+ \tau_X\norm{x_{n+1}-x_{n}}_{ X}^2\notag
	\\&\qquad = g([\nabla_x J(\xi)-\nabla_xJ_{n,n}](x_{n+1}-x_{n}))+ \tau_X\norm{x_{n+1}-x_{n}}_{ X}^2\notag
	\\&\qquad \leq \op{Lip}_{X,g}(\nabla_x J(\cdot,y_n))\norm{x_{n+1}-x_{n}}_{ X}^2+ \tau_X\norm{x_{n+1}-x_{n}}_{ X}^2 \notag
	\\&\qquad \leq \tau_x\norm{x_{n+1}-x_{n}}_{ X}^2 \label{gc2}
	\end{align}
	By equivalent argument,
	\begin{equation}
	g(J_{n+1,n+1}) - g(J_{n,n+1}+\nabla_yJ_{n+1,n}(y_{n+1}-y_{n}))+ \tau_Y\norm{y_{n+1}-y_{n}}_{ Y}^2 \leq \tau_y\norm{y_{n+1}-y_n}_{ Y}^2 \label{gc4}
	\end{equation}
	Summing Equations \eqref{gc1}-\eqref{gc4} gives
	\[ E_{n+1,n+1} + \tau_X\norm{x_{n+1}-x_n}_{ X}^2+\tau_Y\norm{y_{n+1}-y_n}_{ Y}^2\leq E_{n,n}\]
	This immediately gives the monotone decrease property of $E_{n,n}$. If we also sum this over $n$ then we achieve the final statement of the theorem:
	\[ \sum_{n=1}^\infty\tau_X\norm{x_{n+1}-x_n}_{ X}^2+\tau_Y\norm{y_{n+1}-y_n}_{ Y}^2\leq E_{0,0} - \lim E_{n,n}\leq E_{0,0}\]
	\end{proof}

	\subsection{Convergence to Critical Points}
	First we follow the work by Drusvyatskiy et. al. \cite{Drusvyatskiy2016a} we shall define criticality in terms of the slope of a function.
	\begin{definition}\label{critical point}
	We shall say that $x_*$ is a \oldtext{first order minimum} \newtext{critical point} of $F\colon X\to \F R$ if 
	\[ |\partial F(x_*)|= 0\]
	where we define \newtext{the slope of $F$ at $x_*$ to be}
	\[|\partial F(x_*)| = \limsup_{dx\to 0} \frac{\max(0,F(x_*)-F(x_*+dx))}{\norm{dx}}\]
	\end{definition}
	The first point to note is that this definition generalises the concept of a first order critical point for both smooth functions and convex functions (in terms of the convex sub-differential). In particular
	\begin{align*}
	F\in C^1 \implies &|\partial F(x_*)| = \max\left(0,\sup_{\norm{dx}= 1}-\IP{\nabla F(x_*)}{dx}\right) = \norm{\nabla F(x_*)} 
	\\\text{Hence } &|\partial F(x_*)| = 0 \iff \norm{\nabla F(x_*)} = 0\iff \nabla F(x_*) = 0
	\end{align*}
	\begin{align*}
	F\text{ convex, hence } & x^*\in \argmin F \iff F(x_*) \leq F(x_*+dx) \text{ for all } dx
	\\\text{Hence } &\hspace{-1pt}|\partial F(x_*)| = 0 \iff \forall dx, 0\geq \frac{F(x_*)-F(x_*+dx)}{\norm{dx}} \iff x_*\in\argmin F
	\end{align*}

	\oldtext{However, at a point of non-differentiability with $|\partial F|=0$ we can still say a little more than local criticality (a priori a minimum/maximum/saddle point). Due to the $\limsup$ in the definition we can in fact rule out some of these situations in a way that biases us towards minima, this is why we refer to this as a first order minimum rather than simply a first order critical point.} 
	\newtext{For a differentiable function we cannot tell whether a critical point is a local minimum, maximum or saddle point. In general, this is also true for Definition \ref{critical point}, however, at points of non-differentiability there is a bias towards local minima. This can be seen in the following example.}
	\begin{example}
		Consider $F(x) = -\norm{x}$
		\[|\partial F(0)| = \limsup_{dx\to 0}\;\max\left(0,\frac{0+\norm{0+dx}}{\norm{dx}}\right) = 1\]
		\oldtext{ Hence, 0 is not a first order minimum of $F$ which fits with an intuitive understanding of the term. In fact, due to the $\limsup$, if \emph{any} directional derivative is negative then we are not at a first order minimum.}
		\newtext{Hence, 0 is not a critical point of $F$. This bias is due to the $\limsup$ in the definition which detects the strictly negative directional derivatives. This doesn't affect smooth functions as directional derivatives must vanish continuously to 0 about a critical point.}
	\end{example}
	Some more examples are shown in Figure \ref{critical points}.
	\begin{figure}
		\centering
		\begin{subfigure}{.42\textwidth}
			\begin{tikzpicture}
				\node at (0,0) {\includegraphics[width=\textwidth,trim={0 0 300 0}, clip]{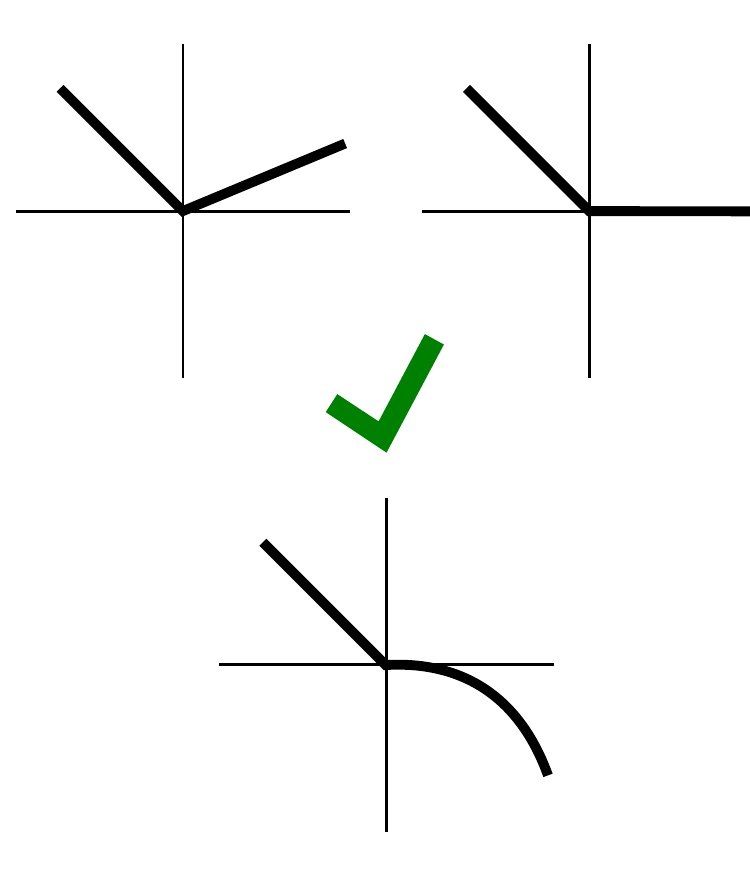}};
				\node at (1.6,-1.0) {$x> 0, \epsilon>0$};
				\node at (1.6,-1.5) {$F(x)=-x^{1+\epsilon}$};
			\end{tikzpicture}
			\caption{\newtext{Examples of Critical Points}}
		\end{subfigure}
		\hfill
		\begin{subfigure}{.42\textwidth}
			\begin{tikzpicture}
				\node at (0,0) {\includegraphics[width=\textwidth,trim={300 0 0 0}, clip]{curves}};
				\node at (1.3,-1.0) {$x>0, \epsilon\leq0$};
				\node at (1.3,-1.5) {$F(x)=-x^{1+\epsilon}$};
			\end{tikzpicture}
			\caption{\newtext{Examples of Non-Critical Points}}
		\end{subfigure}
		\caption{Examples of 1D functions where 0 is/isn't a critical point by Definition \ref{critical point}. If a function is piece-wise linear then 0 is a critical point iff each directional derivative is non-negative. If the function can be approximated on an interval of $x>0$ to first order terms by $F(x) = cx^{1+\epsilon}$ then criticality can be characterised sharply. If $c\geq0$ then 0 is always a critical point. If $c<0$ then 0 is a critical point iff $\epsilon>0$, however, 0 is never a local minimum.}\label{critical points}
	\end{figure}
	Now we shall show that our iterative sequence converges to a critical point in this sense.
	\begin{lemma}\label{critical convergence}
		If $(x_n,y_n)$ are as given \oldtext{above} \newtext{by Algorithm \ref{alg}} and $X, Y$ are finite dimensional spaces then every limit point of $(x_n,y_n)$, e.g. $(x_*,y_*)$, is a critical point of $E$ in each coordinate direction. i.e. \[|\partial_xE(x_*,y_*)| = |\partial_yE(x_*,y_*)| = 0\]
	\end{lemma}
	\begin{remark}\hfill
		\begin{itemize}
			\item \oldtext{Our assumption of finite dimensionality} \newtext{Finite dimensionality of $X$ and $Y$} accounts for what is referred to as `Assumption 3' in \cite{Ochs2017} and is some minimal condition which ensures that the limit is also a stationary point of our iteration (Equations \eqref{xstep}-\eqref{ystep}).
			\item \newtext{The condition for finite dimensionality, as we shall see, does not directly relate to the non-convexity. The difficulty of showing convergence to critical points in infinite dimensions is common across both convex \cite{Chambolle2011} and non-convex \cite{iPalm,Ochs2017} optimisation.}
		\end{itemize} 
	\end{remark}

	\begin{proof}
	First we recall that, in finite dimensional spaces, convex functions are continuous on the relative interior of their domain \cite{Duchi2017}. Also note that by our \oldtext{definition} \newtext{choice} of $\tau_x$ \newtext{in Lemma \ref{critical convergence}}, for all $x,x',y$ we have
	\begin{align*}
	|g(J(x,y)+\nabla_x J(x,y)(x'-x)) &- g(J(x',y))|
	\\&\leq |g([J(x,y)-J(x',y)]-\nabla_x J(x,y)(x'-x))| 
	\\&= |g(\nabla_x J(\xi,y)(x'-x)-\nabla_x J(x,y)(x'-x))|
	\\&\leq \tau_x\norm{\xi-x}_X\norm{x'-x}_X
	\\&\leq \tau_x\norm{x'-x}_X^2
	\end{align*}
	where existence of such $\xi$ is given by the Mean Value theorem. Hence, for all $x$ we have
	\begin{align*}
		E(x_{n+1},y_n) & = f_{n+1,n} + g(J_{n+1,n})
		\\&\leq f_{n+1,n} + g(J_{n,n}+\nabla_xJ_{n,n}(x_{n+1}-x_n)) + \tau_x\norm{x_{n+1}-x_n}_X^2 
		\\&\leq f(x,y_n)+g(J_{n,n}+\nabla_xJ_{n,n}(x-x_n))+\tau_x\norm{x-x_n}_X^2
		\\&\leq f(x,y_n)+g(J(x,y_n))+2\tau_x\norm{x-x_n}_X^2 
		\\&= E(x,y_n) + 2\tau_x\norm{x-x_n}_X^2 
	\end{align*}
	where the first and third inequality are due to the condition shown above and the second is due to the definition of $x_{n+1}$ in \eqref{xstep}. Finally, by continuity of $f$, $J$ and $g$ we can take limits on both sides of this inequality:
	\begin{equation}\label{conv to stationary}
	\implies E(x_*,y_*) \leq E(x,y_*) +2\tau_x\norm{x-x_*}_X^2 \text{ for all } x
	\end{equation}
	This completes the proof for $x_*$ as 
	\[|\partial_xE(x_*,y_*)| = \limsup_{x\to x_*} \max\left(0,\frac{E(x_*,y_*)-E(x,y_*)}{\norm{x_*-x}_X}\right) \leq \limsup 2\tau_x\norm{x-x_*}_X = 0\]
	The proof for $y_*$ follows by symmetry.
	\end{proof}
	\begin{remark}\hfill
		\begin{itemize}
			\item The important line in this proof, and where we need finite dimensionality, is being able to pass to the limit for \eqref{conv to stationary}. In the general case we can only expect to have $(x_n,y_n)\rightharpoonup (x_*,y_*)$, typically guaranteed by choice of regularisation functionals as in our Theorem \ref{conv thm}. In this reduced setting the left hand limit of \eqref{conv to stationary} still remains valid, \[E(x_*,y_*) \leq \liminf E(x_{n+1},y_n) \text{ by weak lower semi-continuity.}\]
			However, on the right hand side we require:
			$$\oldtext{f(x,y_*) + g(J(x,y_*)) + 2\tau_x\norm{x-x_*}_X^2\stackrel{?}\geq \lim f(x,y_n) + g(J(x,y_n)) + 2\tau_x\norm{x-x_n}_X^2}$$
			\newtext{\[ \lim E(x,y_n) + 2\tau_x\norm{x-x_n}_X^2 \leq E(x,y_*) + 2\tau_x\norm{x-x_*}_X^2\]}
			In particular, we already require $\norm{x-x_n}_X$ to be weakly upper semi-continuous. Topologically, this is the statement that weak and norm convergence are equivalent which will not be the case in most practical (infinite dimensional) examples.
			\item \newtext{The properties we derive for $(x_*,y_*)$ are actually slightly stronger than that of Definition \ref{critical point} which only depends on an infinitesimal ball about $(x_*,y_*)$. However, \eqref{conv to stationary} gives us a quantification for the more global optimality of this point. This is seen in Figure \ref{critical_point}.}
		\end{itemize}
	\end{remark}
	\begin{figure}
		\centering
		\begin{tikzpicture}
		\node at (0,0) {\includegraphics[width=.5\textwidth]{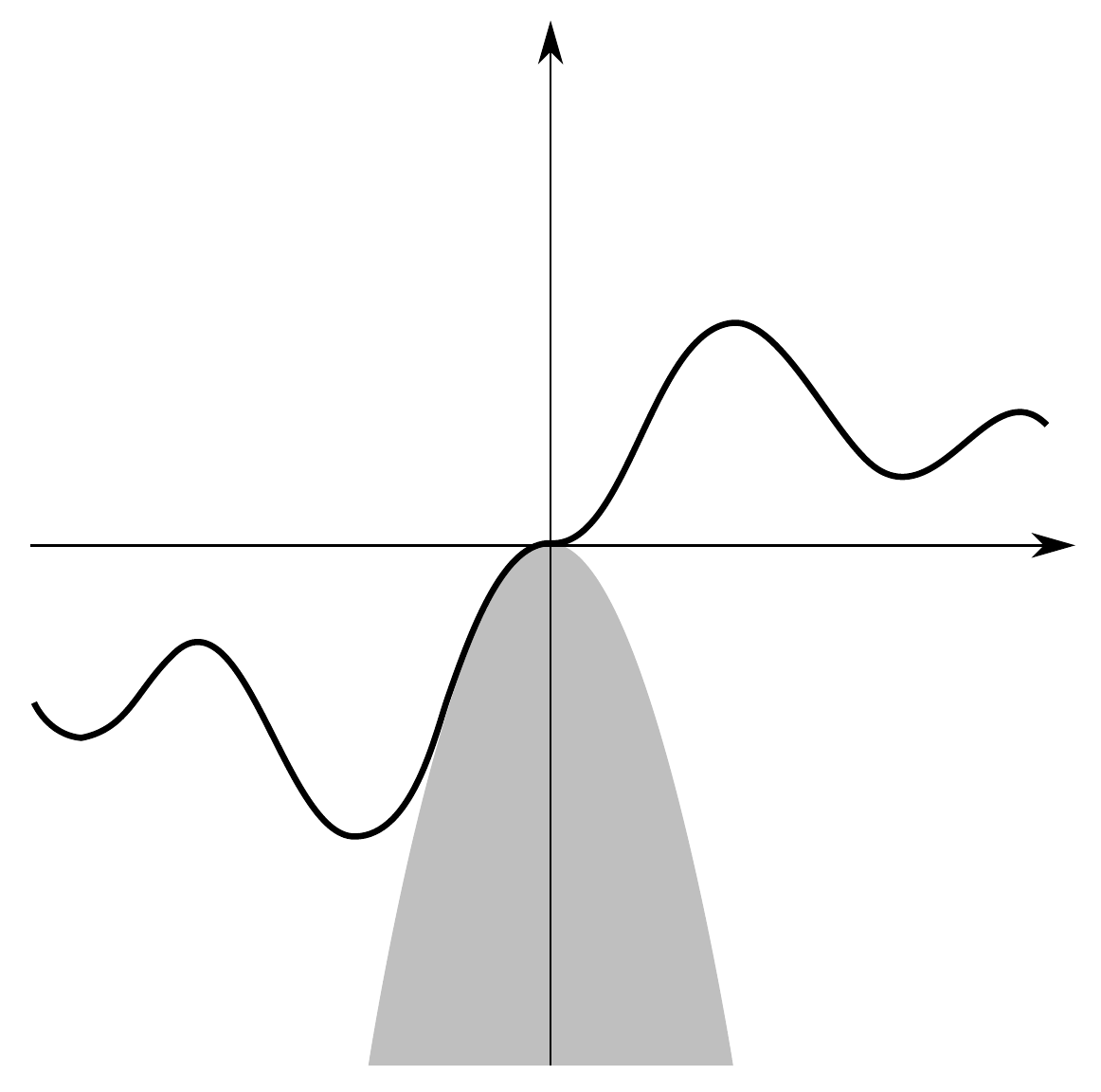}};
		\node at (3,-2.5) {$E(x_*,y_*)-2\tau_x\norm{x-x_*}^2$};
		\node at (4,0) {$x$};
		\node at (1.0,3.0) {$E(x,y_*)$};
		\end{tikzpicture}
		\caption{\oldtext{In general we cannot guarantee that limit points are local minima but we can quantify the speed at which the energy may decay in a neighbourhood. In the case of quadratic penalty functions we have an explicit formula for a supporting cone which bounds the energy from below.} \newtext{Theorem \ref{conv thm} tells us that $(x_*,y_*)$ is a local critical point but does not qualify the globality of the limit point. Equation \eqref{conv to stationary} further allows us to quantify the idea that if a lower energy critical point exists then it must lie far from $(x_*,y_*)$. In particular, it must lie outside of the shaded cone given by the supporting quadratic.}}\label{critical_point}
\end{figure}
	
	\subsection{Proof of Theorem \ref{alternating thm}}
	\begin{proof}
		Fix arbitrary $(x_0,y_0)\in X\times Y$ and $\tau_X,\tau_Y\geq0$. Let $(x_n,y_n)$ be defined as in Algorithm \ref{alg}. By Lemma \ref{global convergence} we know that $\{(x_n,y_n)\st n\in\F N\}$ is contained in a sub-level set of $E$ which in turn must be weakly compact by \eqref{myFunc-start}. The assumption of Lemma \ref{critical convergence} is that we are in a finite dimensional space and so weak compactness is equivalent to norm compactness, i.e. some subsequence of $(x_n,y_n)$ converges in norm. Also by Lemma \ref{critical convergence} we know that the limit point of this sequence must be an appropriate critical point.
	\end{proof}

	\section{Results}
	For numerical results we present two synthetic examples and one experimental dataset from Transmission Electron Tomography. The two synthetic examples are \newtext{discretised at a resolution of $200\times200$ then }simulated using the X-ray transform \newtext{with a parallel beam geometry} sampled at $1^\circ$ intervals over a range of $60^\circ$ \newtext{resulting in a full sinogram of size $287\times180$ and sub-sampled data at $287\times60$}. Gaussian white noise (standard deviation 5\% maximum signal) is then added to give the data. The experimental dataset was acquired with an annular dark field \newtext{(parallel beam)} Scanning TEM modality from which we have 46 projections spaced uniformly in $3^\circ$ intervals over a range of $135^\circ$. \newtext{Because of the geometry of the acquisition, we can treat the original 3D dataset as a stack of 2D and thus extract one of these slices as our example. This 2D dataset is then sub-sampled to 29 projections over $87^\circ$, reducing the size from $173\times45$ to $173\times29$. This results in a reconstruction with $u$ of size $120\times120$ and $v$ of size $173\times180$.} A more detailed description of the acquisition and sample properties of the experimental dataset can be found in \cite{Collins2015}. The code, and data, for all examples is available \footnote{\url{https://github.com/robtovey/2018_Directional_Inpainting_for_Limited_Angle}} under the Creative Commons Attribution (CC BY) license.

	\subsection{Numerical Details}\label{numerical details}
	All numerics are implemented in MATLAB 2016b. The sub-problem for $u$ is solved with a PDHG algorithm \cite{Chambolle2011} while the sub-problem for $v$ is solved using the MOSEK solver via CVX \cite{Mosek,CVX}, the step size $\tau$ is adaptively calculated. The initial point of our algorithm is always chosen to be a good TV reconstruction, i.e. \[u_0 = \argmin_{u\geq 0}\frac12\norm{S\C Ru-b}_2^2 + \lambda\TV(u), \qquad v_0 = \C R u_0\]
	\oldtext{Our construction for the directional regularisation was chosen to be $c_1(\Delta,\Sigma) = 10^{-6}+\frac{\tanh(\Sigma)}{1+(\beta_3\Delta)^2}, c_2(\Delta,\Sigma) = 10^{-6}+\tanh(\Sigma)$ as the property $\partial_\Delta c_1$ is maximally negative for $\Sigma,\Delta\approx0$ which encourages $\Delta$ to grow on regions which are near flat. In particular, this encourages edge continuation in the sinogram when they would otherwise be blurred out. Our synthetic experiments suggested that the parameter choices were relatively intuitive at a fixed resolution and scaling (both synthetic images satisfied $\norm{u}_\infty = 1$) although switching to the synthetic dataset (different resolution, scaling, angular distribution) the optimal parameter values shifted significantly. Despite the large number of tuning parameters, we observed that exact tuning was unnecessary and so most were tuned within a factor of $10^{\pm\sfrac12}$.}
	\newtext{For clarity, we shall restate our full model with all of the parameters it includes. We seek to minimise the functional \eqref{model}:
	\begin{equation*}
	E(u,v) = \frac{1}{2}\norm{\C Ru-v}^2_{\alpha_1} + \frac{\alpha_2}{2}\norm{S\C Ru-b}_2^2 + \frac{\alpha_3}{2}\norm{Sv-b}_2^2 + \beta_1\op{TV}(u) + \beta_2\norm{A_{Ru}\nabla v}_{2,1}
	\end{equation*}
	\vspace{-28pt}
	\begin{align*}
	A_d(x) &= c_1(x|\lambda_1-\lambda_2, \lambda_1+\lambda_2)\vec{e}_1(x)\vec{e}_1(x)^T + c_2(x|\lambda_1-\lambda_2, \lambda_1+\lambda_2)\vec{e}_2(x)\vec{e}_2(x)^T
	\\\text{where } &(\nabla d_\rho\nabla d_\rho^T)_\sigma = \lambda_1 \vec{e}_1\vec{e}_1^T + \lambda_2 \vec{e}_2\vec{e}_2^T \text{ is a pointwise eigenvalue decomposition}
	\\& c_1(x|\Delta,\Sigma) = 10^{-6}+\frac{\tanh(\Sigma(x))}{1+\beta_3\Delta(x)^2},\quad c_2(x|\Delta,\Sigma) = 10^{-6}+\tanh(\Sigma(x))
	\end{align*}
	We chose these particular $c_i$ according to two simple heuristics. If $\Sigma$ is large (steep gradients) then it is likely a region with edges and so the regularisation should be largest but still bounded above. If $\Delta=0^+$ then there is a small or blurred `edge' present and so we want to encourage it to become a sharp jump, i.e. $\partial_\Delta c_1<0$. Theorem \ref{smooth tensor} tells us that choosing $c_i$ as functions of $\Delta^2$ will guarantee accordance with our later convergence results; this leads to our natural choice above. The number of iterations for Algorithm \ref{alg} was chosen to be 200 and 100 for the synthetic and experimental datasets, respectively. To simplify the process of choosing values for the remaining hyper-parameters we made several observations:
	\begin{enumerate}
		\item The choice of $\alpha_i$ and $\beta_i$ appeared to be quite insensitive about the optimum. It is clear within 2-3 iterations whether values are of the correct order of magnitude. After this, values were only tuned coarsely. For example, $\alpha_3$ and $\beta_i$ are optimal within a factor of $10^{\pm\sfrac12}$.
		\item We can chose $\alpha_2=1$ without any loss of generality. In which case, in general, $\beta_1$ should the same order of magnitude as when performing the TV reconstruction to get $u_0, v_0$.
		\item $\alpha_2$ pairs $u$ to the given data and $\alpha_1$ pairs $u$ to the inpainted data, $v$. As such, $\alpha_1$ is spatially varying but should be something like a distance to the non-inpainting region. We chose the binary metric so that $u$ is paired to $v$ uniformly on the inpainting region and not at all outside.
		\item $\DTV$ specific parameters ($\beta_2, \beta_3, \rho, \sigma$) can be chosen outside of the main reconstruction. These were chosen by solving the toy problem: 
		\[\argmin \frac12\norm{v-v_0}_2^2 + \beta_2\norm{A_{v_0}\nabla v}_{2,1}\] which is a lot faster to solve. $\rho>0$ is required for the analysis and so this was fixed at 1. $\sigma$ is a length-scale parameter which indicates the separation between distinct edges. $\beta_3$ relies on the normalisation of the data. As can be seen in Table \ref{param table}, for the two synthetic examples, with same discretisation and scaling, these values are also consistent. The only value which changes is $\beta_2$, as expected, which weights how valid the $\DTV$ prior is for each dataset.
	\end{enumerate}
	It is unclear whether a gridsearch may provide better results although, due to the number of parameters involved, this would definitely take a lot longer and mask some interpretability of the parameters. A further comparison of different choices of the main parameters can be found in the supplementary material.}

	\begin{table}
		\begin{tabular}{|p{120pt}|c|c|c|c|c|c|c|c|}
		\hline 
		& $\alpha_1$ & $\alpha_2$ & $\alpha_3$ & $\beta_1$ & $\beta_2$ & $\beta_3$ & $\rho$ & $\sigma$ \\ 
		\hline 
		Concentric Rings $\qquad$ Phantom & $\frac1{2^2}\1_{\Omega'^c}$ & 1 & $1\times10^{-1}$ & $3\times10^{-5}$ & $3\times10^{3}$ & $10^{10}$ & $1$ & $8$ \\
		\hline 
		Shepp-Logan Phantom & $\frac1{4^2}\1_{\Omega'^c}$ & 1 & $3\times10^{-1}$ & $3\times10^{-5}$ & $3\times10^{2}$ &  $10^{10}$ & $1$ & $8$ \\
		\hline 
		Experimental Dataset (Both sampling ratios) & $\frac1{2^2}\1_{\Omega'^c}$ & 1 & $3\times10^{2}$ & $1\times10^{-5}$ & $3\times10^{1}$ &  $10^6$ & $1$ & $0$ \\
		\hline 
		\end{tabular} 
		\caption{Parameter choices for numerical experiments. \newtext{Each algorithm was run for 300 iterations}}\label{param table}
	\end{table}	

	\subsection{Canonical Synthetic Dataset}
	This example shows two concentric rings. This is the canonical example for our model because the exact sinogram is perfectly radially symmetric which should trivialise the directional inpainting procedure, even with noise present. As can clearly be seen in Figure \ref{donut res}, the TV reconstruction is poor in the missing wedge direction which can be seen as a blurring out of the sinogram. By enforcing better structure in the sinogram, our proposed joint model is capable of extrapolating these local structures from the given data domain to recover the global structure and gives an accurate reconstruction.
	\begin{figure}
		\centering{\includegraphics[width=.95\textwidth,trim={150 80 80 10},clip]{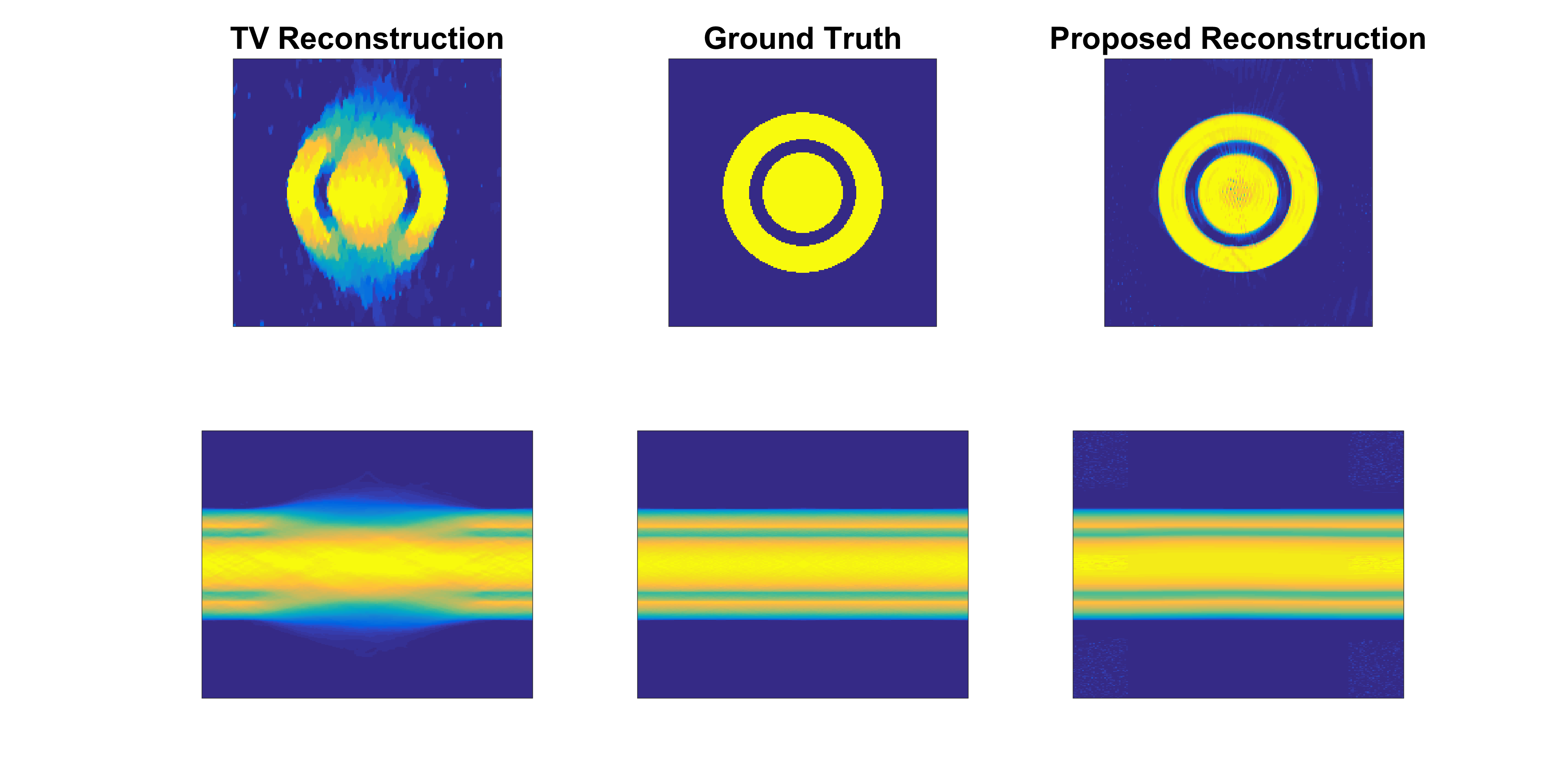}}
		\caption{Canonical synthetic example. Top row shows the reconstructions, $u$, while the bottom row shows the reconstructed sinogram, $v$.}\label{donut res}
	\end{figure}
	\subsection{Non-Trivial Synthetic Dataset}
	This example shows the modified Shepp-Logan phantom which is built up as a sum of ellipses. This example has a much more complex geometry although the sinogram still has a clear geometry. In Figure \ref{shepp res} we see that the largest scale feature, the shape of the largest ellipse, is recovered in our proposed reconstruction with minimal loss of contrast in the interior. One artefact we have not been able to remove is the two rays extending from the top of the reconstructed sample. Looking more closely we found that it was due to a small misalignment of the edge at the bottom of the sinogram as it crosses between the data to the inpainting region. Numerically, this happens because of the convolutions which take place inside the directional TV regularisation functional. Having a non-zero blurring is essential for regularity of the regularisation (Theorem \ref{smooth tensor}) but the effect of this is that it does not heavily penalise misalignment on such a small scale. This means that at the interface between the fixed data-term there is a slight kink, the line is continuous but not $C^1$. The effect of this on the reconstruction is the two lines which extend from the sample at this point. \newtext{Looking at quantitative measures, the PSNR value rises from 17.33 to 17.36 whereas the SSIM decreases from 0.76 to 0.62, from TV to the proposed reconstruction, respectively. These measures are inconclusive and the authors feel that they fail to balance the improvement to global geometry verses more local artefacts in the reconstructions.}	
	\begin{figure}
		\centering{\includegraphics[width=.95\textwidth,trim={150 80 55 10},clip]{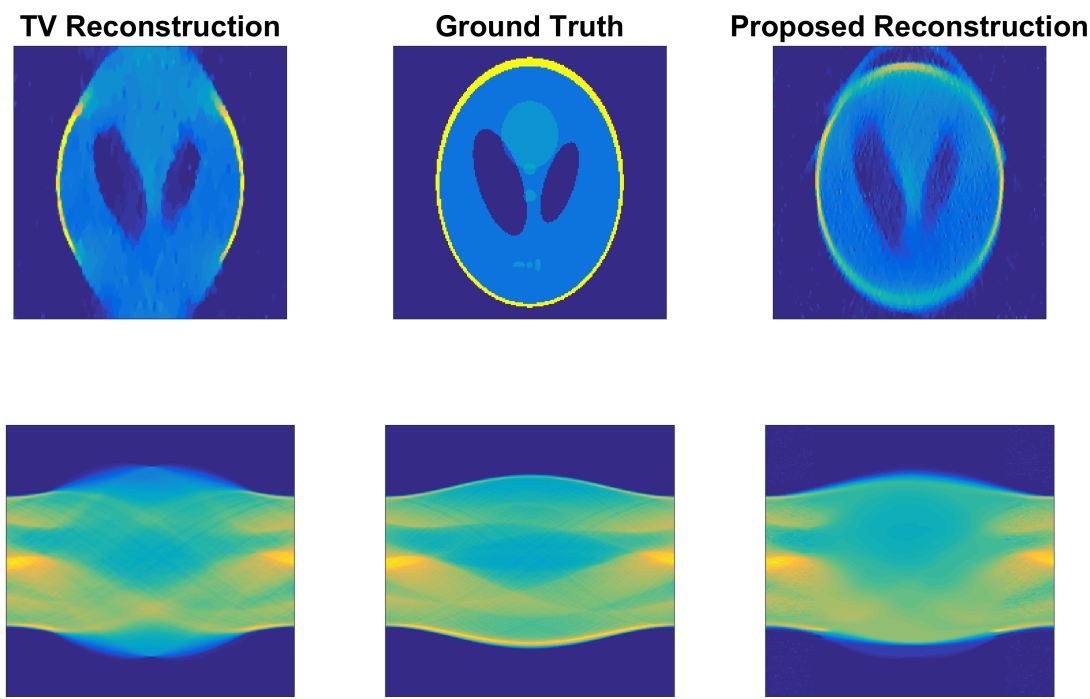}}
		\caption{Non-trivial synthetic example of the modified Shepp-Logan phantom. Top row shows the reconstructions, $u$, while the bottom row shows the reconstructed sinogram, $v$. We regain the large-scale geometry of the shape without losing much of the interior features.}\label{shepp res}
	\end{figure}
	
	\subsection{Experimental Dataset}
	\begin{figure}
		\begin{tikzpicture}
		\node at (0,0) {\includegraphics[width=.6\textwidth,trim={100 60 80 15},clip]{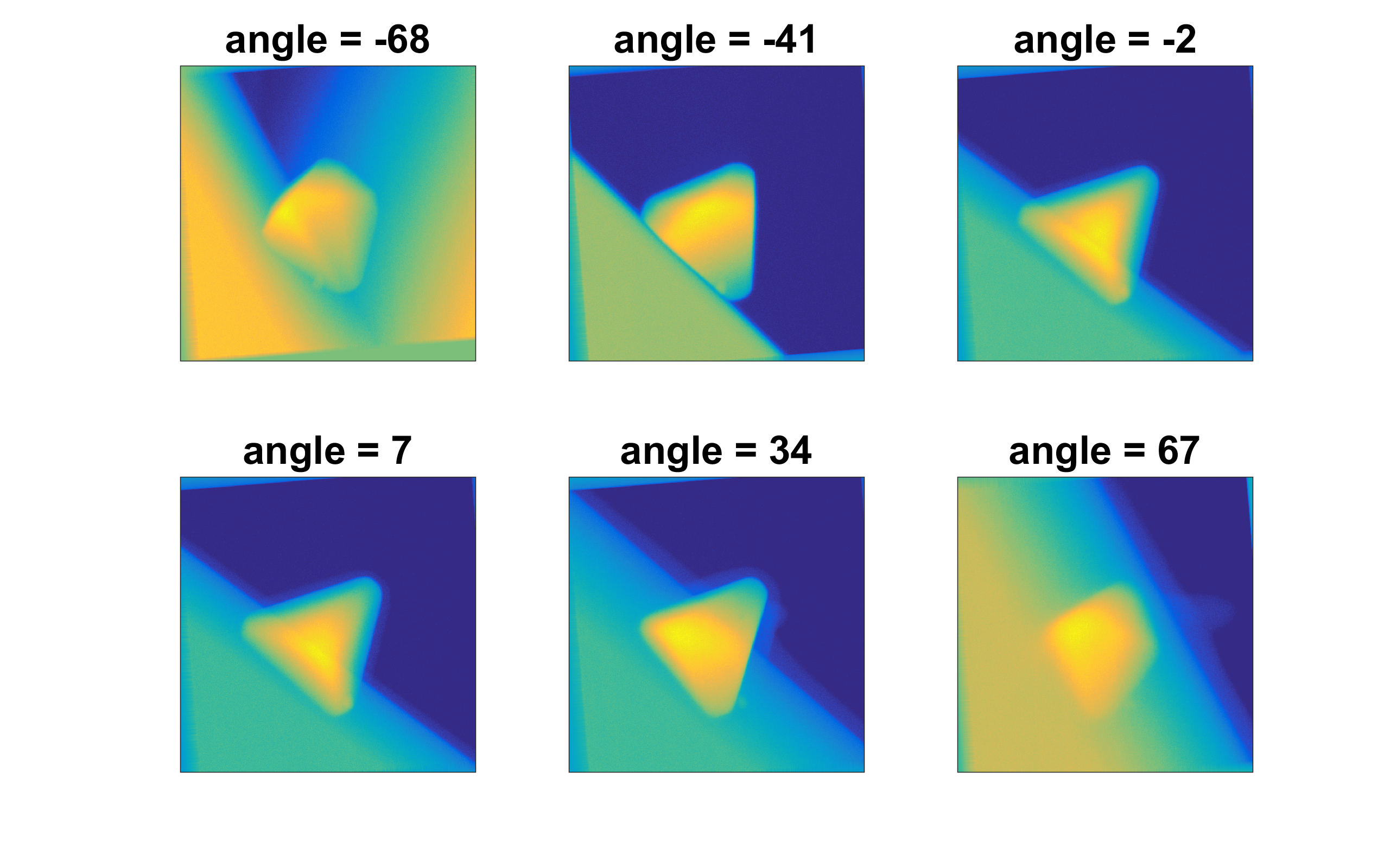}};
		\node at (.5\textwidth,0) {\includegraphics[width=.29\textwidth,trim={50 90 30 0},clip]{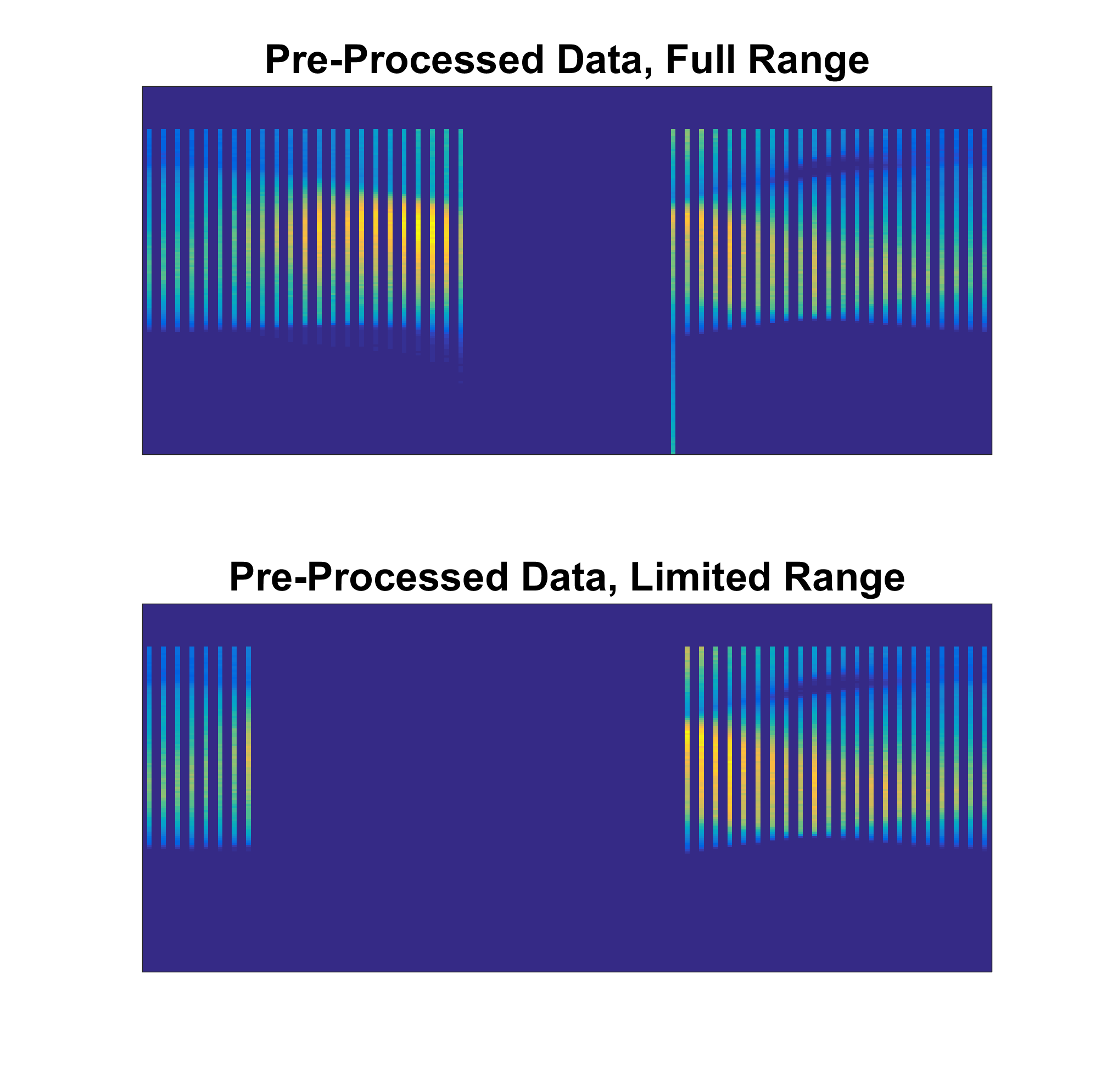}};
		\draw [->,color=red] (1,-1.2) -- (.6,-1.7);
		\draw [->,color=red] (4.3,-1.2) -- (4.2,-1.7);
		\end{tikzpicture}
		\caption{Raw data for Transmission Electron microscopy example. Projections at large angles, e.g. $-68^\circ$,  \oldtext{contain views of supporting grid as was described in Figure2 suggesting that a more limited angle acquisition is desirable. The plane surface will also cause artefacts in the reconstruction as it continues outside the field of view and so more of it is visible in some projections than others. This violates the X-ray modelling assumption that outside the field of view is vacuum. } \newtext{show the presence of the sample holder which violates the X-ray modelling assumption that outside of the region of interest is vacuum. If the violation is too extreme then this can cause strong artefacts in reconstructions and so the common action is to discard such data. The plane surface also violates this model but is relatively weak at low angles and so will cause weaker artefacts.} A source of noise in this acquisition is that over time the surface becomes coated with carbon. This is first visible as a thin film at $-2^\circ$ and progressively gets thicker through the remaining projections. At $34^\circ$ we see a bump of carbon appear on the top right edge. \oldtext{Because of these factors, there is  both a limited angle pressure and pressure for only few projections.} After pre-processing, \oldtext{we have the full dataset } \newtext{we extract a 2D slice of all projections to form the full range} as shown top right artificially sub-sample to compare TV with our proposed reconstruction method.}
		\label{real data}
	\end{figure}
	\begin{figure}
		\centering\includegraphics[width=.8\textwidth,trim={120 50 55 10},clip]{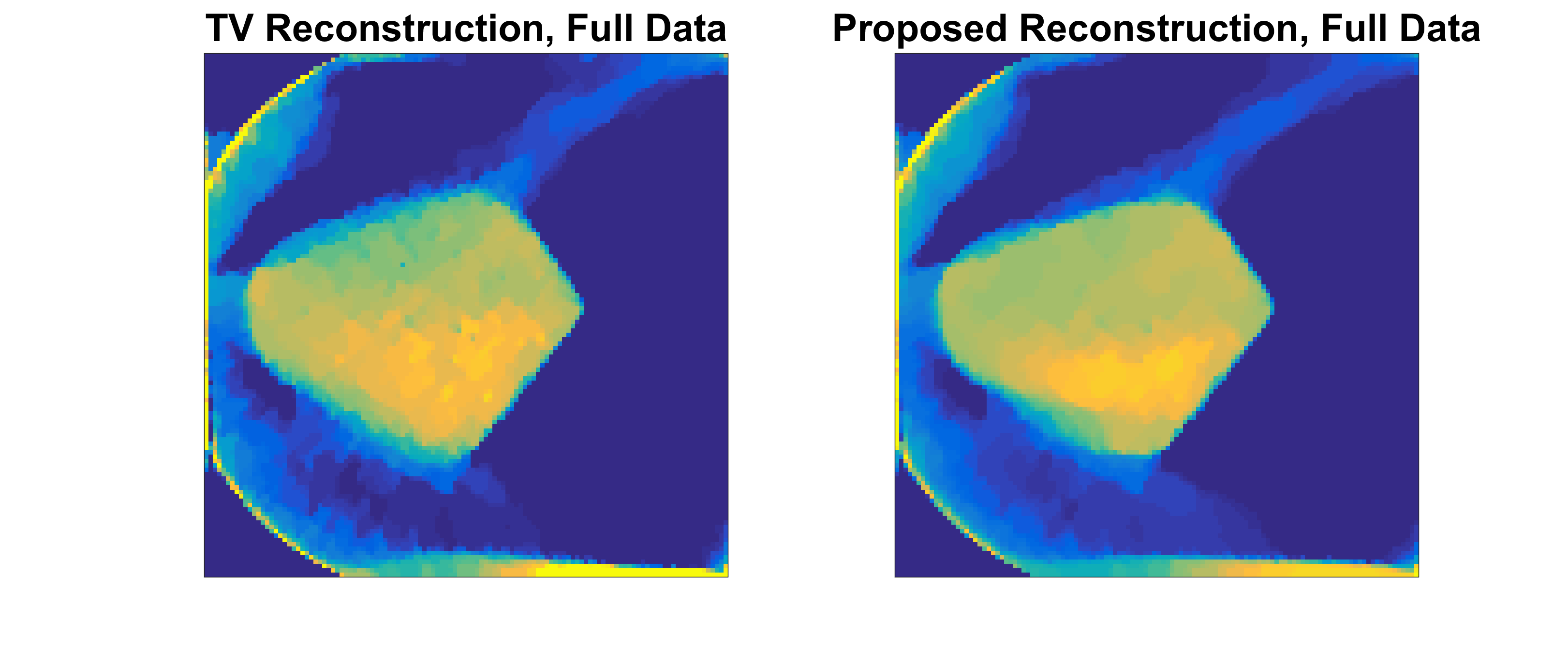}
		\centering\includegraphics[width=.8\textwidth,trim={120 50 55 10},clip]{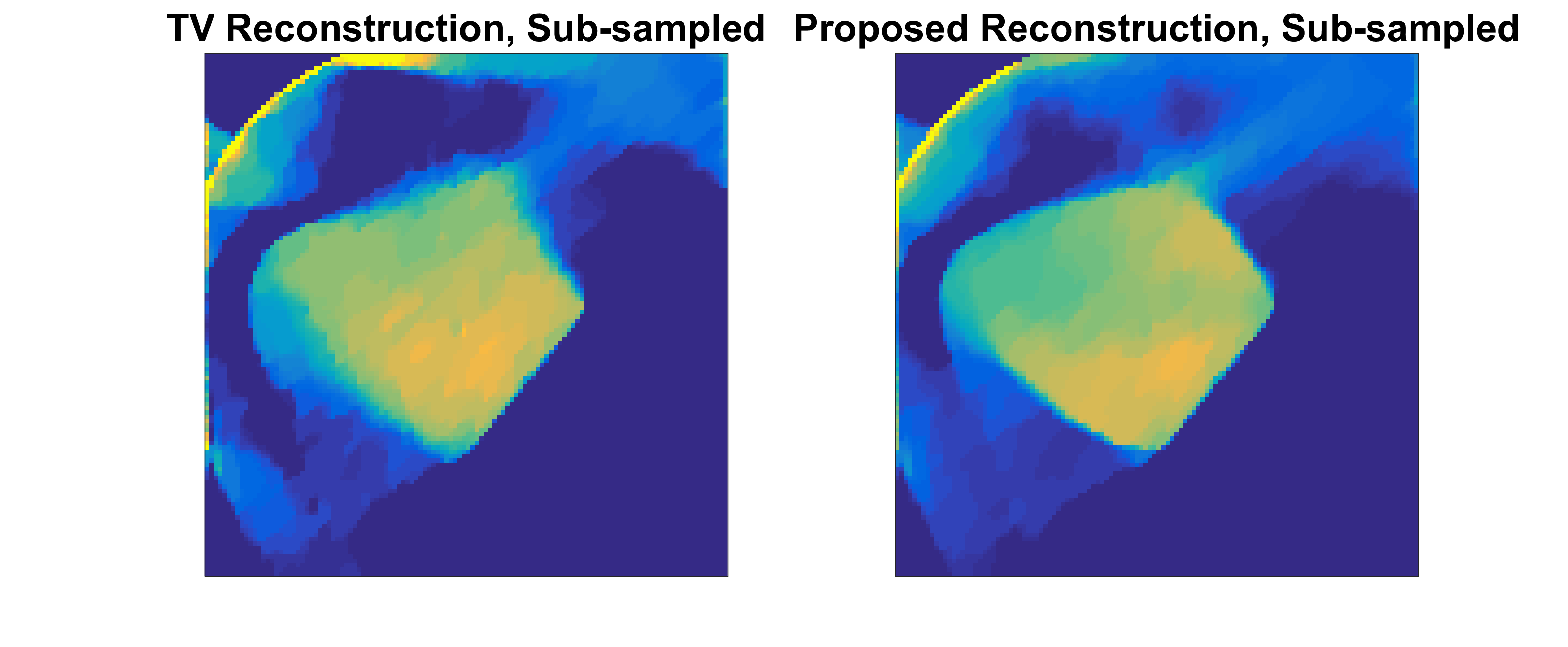}
		\caption{Reconstructions from a slice of the experimental data. We have chosen the slice half-way down through the projections shown in Figure \ref{real data} to coincide with one of the rounded corners. The arc artefact was an anticipated consequence due to the out-of-view mass, the pre-processing has simply reduced the intensity. Proposed reconstructions consistently show better homogeneity inside the particle and sharper boundaries. The missing angles direction is the bottom-left to top-right diagonal where we see most error in each reconstruction, in particular, the blurring of the top right corner of the particle is a limited angle artefact.}
		\label{real TV}
	\end{figure}
	\begin{figure}
		\centering\includegraphics[width=.8\textwidth,trim={100 10 30 10},clip]{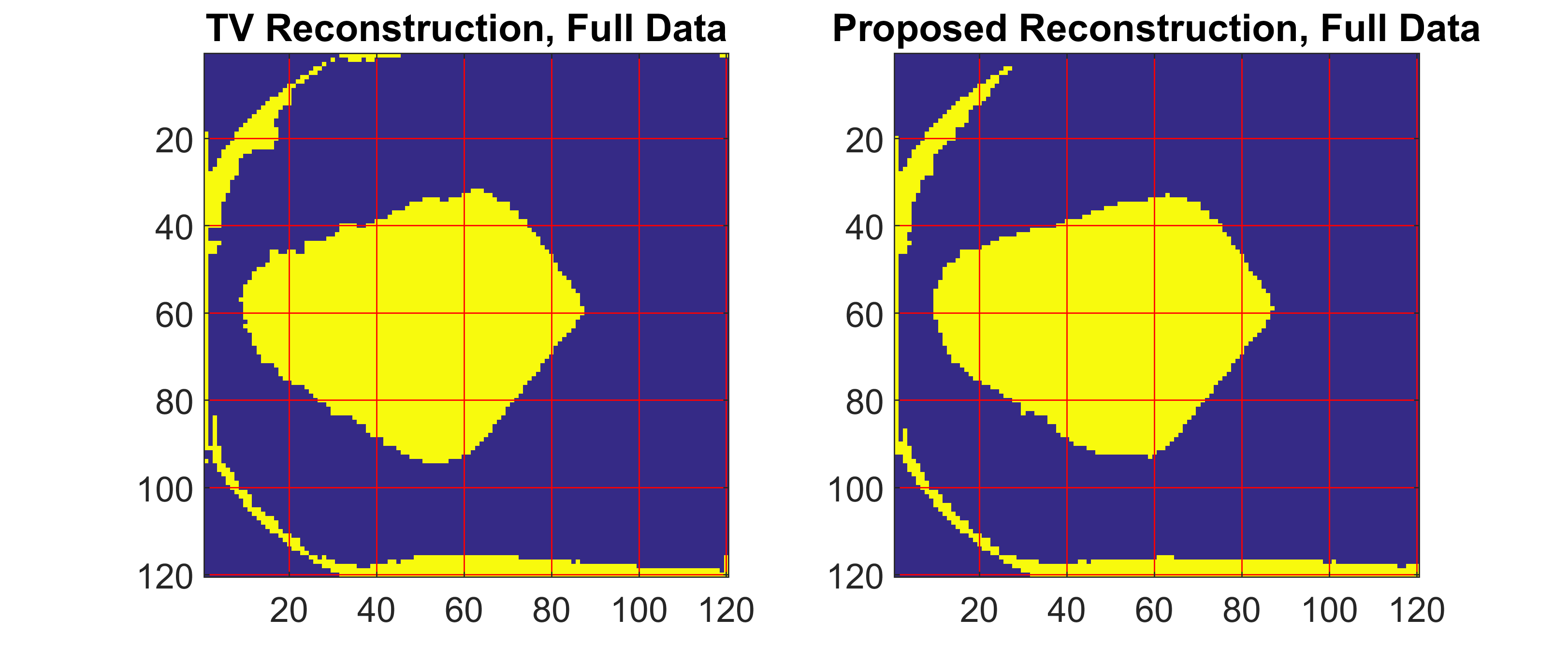}
		\centering\includegraphics[width=.8\textwidth,trim={100 10 30 10},clip]{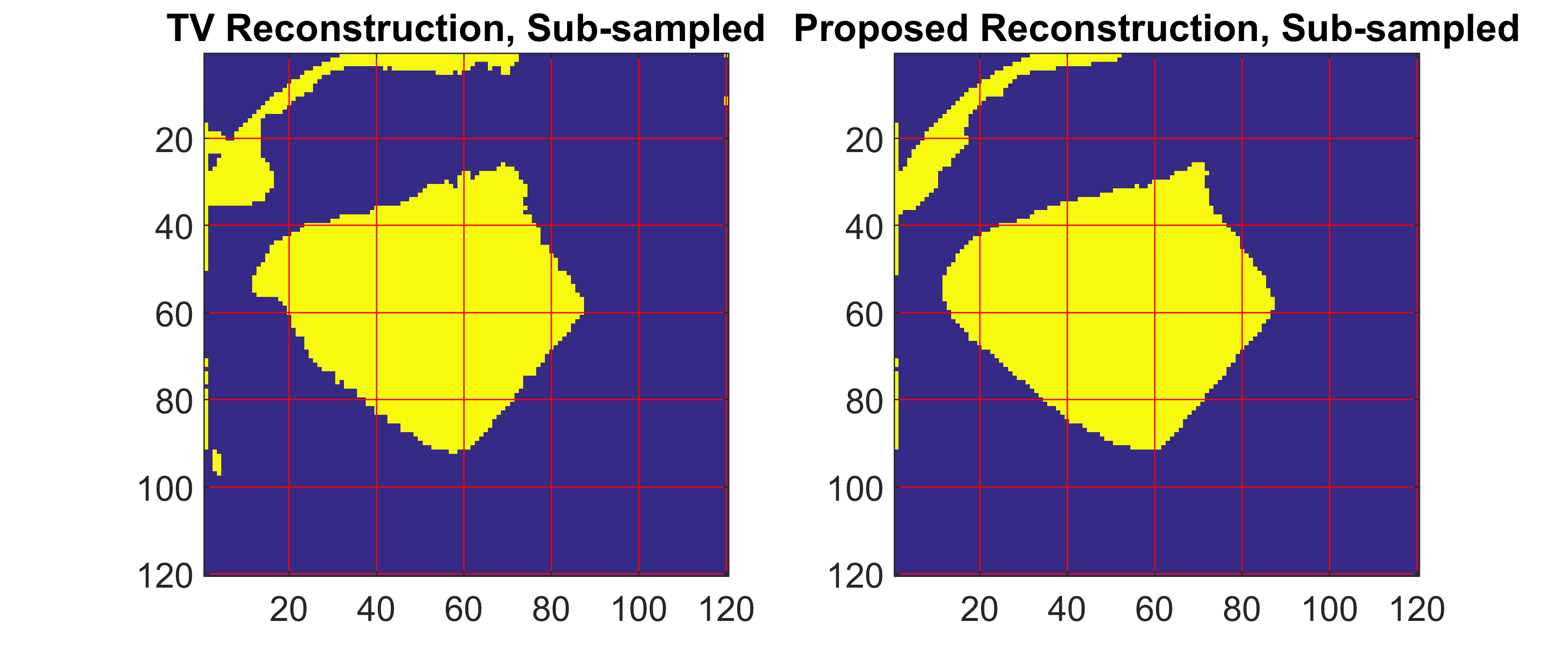}
		\caption{Comparison between each reconstruction after thresholding. The geometrical properties of interest are that each edge should be linear, the left hand corner is rounded and the remaining corners are not. The particle of interest is homogeneous so thresholding the images should emphasise this in a way which is very unsympathetic to blurred edges. Again, the top right corner of each particle in the sub-sampled reconstructions coincides with the exacerbated missing wedge direction and so we expect each reconstruction to make some error here.}
		\label{real comp}
	\end{figure}

	The sample is a silver bipyramidal crystal placed on a planar surface, and the challenges of this dataset are shown in Figure \ref{real data}. 
	\oldtext{We immediately see that the wide angle projections have large artefacts from the planar surface which masks the sample at these high angles. This suggests that a more limited angle acquisition is desirable. Another artefact is that over time each surface becomes coated with carbon. This is a necessary consequence of the sample preparation and this coating is known to occur during the microscopy. The result of these two factors is that we find a pressure to use both a limited angular range and the earliest acquired projections. Because of this, in numerical experiments we compare both TV and our proposed reconstruction using only $\sfrac34$ of the available data, 30 projections over an $87^\circ$ interval, with a bias towards earlier projections. The final artefact in the data is that there is mass seen in some of the projections which cannot be represented within the reconstruction volume. This breaks the X-ray model assumption that the field of view is surrounded by a vacuum and thus will also lead to reconstruction artefacts. To minimise this problem we shall perform preliminary background subtraction which will minimise the total mass contributing to this violation. Effectively, we pretend our detector was narrower such that we have dataonly in a small neighbourhood of the region of interest.}
	\newtext{We immediately see that the wide angle projections have large artefacts which produces a very low signal to noise ratio. Another issue present is that there is mass seen in some of the projections which cannot be represented within the reconstruction volume. Both of these issues violate the simple X-ray model that is used. Exact modelling would require estimation of parameters which are not available a priori and so the preferred acquisition is one which automatically minimises these modelling errors. Another artefact is that over time each surface becomes coated with carbon. This is a necessary consequence of the sample preparation and this coating is known to occur during the microscopy. The result of modelling errors and time dependent noise is to prefer an acquisition with limited angular range and earliest acquired projections. Because of this, in numerical experiments we compare both TV and our proposed reconstruction using only $\sfrac34$ of the available data, 29 projections over an $87^\circ$ interval, with a bias towards earlier projections. The artefacts due to the out-of-view mass are unavoidable but we can perform some further pre-processing to minimise the effect. In particular, if we shrink the field of view of the detector then the `heaviest' part of the data will be the particle of interest and the model violations will be relatively small, increasing the signal to noise ratio.}
	This can be seen as the sharp horizontal cut-off in the pre-processed sinograms seen on the right of Figure \ref{real data}. The effect of this on the reconstruction is going to be that there is a thin ring of mass placed at the edge of the \newtext{(shrunken)} detector view which will be clearly identifiable in the reconstruction. As a ground truth approximation we shall use a TV reconstruction on the full data for the location of the particle alongside prior knowledge of this sample for more precise geometrical features. We also note that the particle should be very homogeneous so this is another example where we expect the TV reconstruction to be very good. 
	
	The sample is a single crystal of silver and so we know it must have very uniform intensity and we are interested in locating the sharp facets which bound the crystal \cite{Collins2015}. In Figure \ref{real TV} we immediately see that the combination of homogeneity and sharp edges is better reconstructed in our proposed reconstruction. \oldtext{In Figure \ref{real comp} we show the thresholded images to visually emphasise the geometrical properties of each reconstruction. This is particularly helpful when locating boundaries or estimating interior angles of the particle.} \newtext{Because we expect the reconstruction to be constant on the background and the particle, thresholding the reconstruction allows us to easily locate the boundaries and estimate interior angles of the particle. Figure \ref{real comp} shows such images where the threshold is chosen to be the approximate midpoint of these two levels.} We see that the proposed reconstruction consistently has less jagged edges and the left hand corner is better curved, as is consistent with our knowledge of the sample. Looking back at the full colour images we see that this is a result of lack of sharp decay at the boundary and homogeneity inside the sample. Looking for location error we see the biggest error in both TV and joint reconstruction is on the bottom-left edge where both reconstructions pull the line inwards. However, looking particularly at points $(40,80)$ and $(20,60)$, we see that this was less severe in the proposed method. The other missing wedge artefact is in the top right corner which has been extended in both reconstructions although it is thinner in the proposed reconstruction. This indicates that it was better able to continue the straight edges either side of the corner and the blurring in the missing wedge direction is more localised than in the TV reconstruction. Overall, we see see that the proposed reconstruction method is much more robust to an decrease in angular sampling range.

	\section{Conclusions and Outlook}
	In this paper we have presented a novel method for tomographic reconstructions in a limited angle scenario along with a numerical algorithm with convergence guarantees. We have also tested our approach on synthetic and experimental data and shown consistent improvement over alternative reconstruction methods. Even when the X-ray transform model is noticeably violated, as with our experimental data, we still better recover boundaries of the reconstructed sample.

	There are three main directions which could be explored in future. Firstly, we think there is great potential to apply our framework to other applications, such as in tomographic imaging with occlusions and heavy metal artefacts where the inpainting region is much smaller \cite{Kostler2006, Zhang2011}. Secondly, we would like to find an alternative numerical algorithm with either faster practical convergence or one which is more capable of avoiding local minima in this non-convex landscape. 
	\oldtext{Finally, we would like to explore the potential for an alternative regularisation functional on the sinogram which is better able to encode structure of the X-ray transform. At the moment we treat sinograms as simple functions on $\F R^2$ ignoring the fact that it must also lie in the image of $\C R$. The discrepancy between the two is seen in Figure \ref{shepp res} as the reconstruction with our proposed method introduces an artefact due to the model mismatch with the X-ray transform. The question of forming regularisation functions consistent with the global structure of a sinogram naturally leads us to the field of micro-local analysis. From this viewpoint the inpainting problem is separated into what are known as visible and invisible singularities; formalising the idea that edges at particular orientations are `visible' but noisy in the limited angle data while the rest must be inpainted. Ideally, our two regularisers would operate independently on just the visible and invisible domains respectively whereas, at the moment, this separation cannot be made.}
	\newtext{Finally, we would like to explore the potential for an alternative regularisation functional on the sinogram which is better able to treat visible and invisible singularities, denoising and inpainting problems, independently. At the moment, the TV prior alone can reconstruct visible singularities well however, introducing a sinogram regulariser currently improves on the invisible region at the expensive of damaging the visible.} Overall, we feel that this presents the natural progression for the current work although it remains unclear how to regularise these invisible singularities.
	
	\section*{Acknowledgements}
	RT would like to thank Ozan \"Oktem for valuable discussion on the application of micro-local analysis in this setting. RT acknowledges support from the EPSRC grant EP/L016516/1 for the University of Cambridge Centre for Doctoral Training, the Cambridge Centre for Analysis. MB acknowledges support from the Leverhulme Trust Early Career Fellowship Learning from mistakes: `A supervised feedback-loop for imaging applications', and the Isaac Newton Trust. CBS acknowledges support from the Leverhulme Trust project on `Breaking the non-convexity barrier', EPSRC grant Nr. EP/M00483X/1, the EPSRC Centre Nr. EP/N014588/1, the RISE projects CHiPS and NoMADS, and the Alan Turing Institute. CBS, MB and RT also acknowledge the Cantab Capital Institute for the Mathematics of Information. \newtext{CB acknowledges support from the EU RISE project NoMADS, the PIHC innovation fund of the Technical Medical Centre of UT and the Dutch 4TU programme Precision Medicine.} MJL acknowledges financial support from the Netherlands Organization for Scientific Research (NWO), project 639.073.506. SMC acknowledges support from the Henslow Research Fellowship at Girton College, Cambridge. RKL acknowledges support from a Clare College Junior Research Fellowship. PAM acknowledges EPSRC grant EP/R008779/1.
	
	\nocite{Fanelli2008}
	\bibliography{library}

\begin{thebibliography}{10}

\bibitem{Ehrhardt2015}
Matthias~J. Ehrhardt, Kris Thielemans, Luis Pizarro, David Atkinson,
  S{\'{e}}bastien Ourselin, Brian~F. Hutton, and Simon~R. Arridge.
\newblock {Joint reconstruction of PET-MRI by exploiting structural
  similarity}.
\newblock {\em Inverse Problems}, 31(1):015001, 2015.

\bibitem{Leary2013}
Rowan~K. Leary, Zineb Saghi, Paul~A. Midgley, and Daniel~J. Holland.
\newblock {Compressed sensing electron tomography}.
\newblock {\em Ultramicroscopy}, 131:70--91, aug 2013.

\bibitem{Kubel2005}
Christian K{\"{u}}bel, Andreas Voigt, Remco Schoenmakers, Max Otten, David Su,
  Tan-Chen Lee, Anna Carlsson, and John Bradley.
\newblock {Recent advances in electron tomography: TEM and HAADF-STEM
  tomography for materials science and semiconductor applications.}
\newblock {\em Microscopy and Microanalysis}, 11(5):378--400, oct 2005.

\bibitem{Zhao2013}
Gongpu Zhao, Juan~R. Perilla, Ernest~L. Yufenyuy, Xin Meng, Bo~Chen, Jiying
  Ning, Jinwoo Ahn, Angela~M. Gronenborn, Klaus Schulten, Christopher Aiken,
  and Peijun Zhang.
\newblock {Mature HIV-1 capsid structure by cryo-electron microscopy and
  all-atom molecular dynamics}.
\newblock {\em Nature}, 497(7451):643--646, may 2013.

\bibitem{Kalender2006}
Willi~A. Kalender.
\newblock {X-ray computed tomography}.
\newblock {\em Physics in Medicine and Biology}, 51(13):29--43, jul 2006.

\bibitem{Cnudde2013}
Veerle Cnudde and Matthieu~Nicolaas Boone.
\newblock {High-resolution X-ray computed tomography in geosciences: A review
  of the current technology and applications}.
\newblock {\em Earth-Science Reviews}, 123:1--17, aug 2013.

\bibitem{Kawase2007}
Noboru Kawase, Mitsuro Kato, Hideo Nishioka, and Hiroshi Jinnai.
\newblock {Transmission electron microtomography without the ``missing wedge''
  for quantitative structural analysis}.
\newblock {\em Ultramicroscopy}, 107(1):8--15, jan 2007.

\bibitem{Zhang2006}
Yiheng Zhang, Heang~Ping Chan, Berkman Sahiner, Jun Wei, Mitchell~M. Goodsitt,
  Lubomir~M. Hadjiiski, Jun Ge, and Chuan Zhou.
\newblock {A comparative study of limited-angle cone-beam reconstruction
  methods for breast tomosynthesis}.
\newblock {\em Medical Physics}, 33(10):3781--3795, sep 2006.

\bibitem{Quinto1993}
E~T Quinto.
\newblock {Singularities of the X-ray transform and limited data tomography in
  {\{}R{\^{}}2{\}} and {\{}R{\^{}}3{\}}}.
\newblock {\em SIAM Journal on Mathematical Analysis}, 24(5):1215--1225, 1993.

\bibitem{Frikel2013a}
J{\"{u}}rgen Frikel and Eric~Todd Quinto.
\newblock {Characterization and reduction of artifacts in limited angle
  tomography}.
\newblock {\em Inverse Problems}, 29(12):21, 2013.

\bibitem{Katsevich1997}
Alexander~I. Katsevich.
\newblock {Local tomography for the limited-angle problem}.
\newblock {\em Journal of Mathematical Analysis and Applications},
  213(1):160--182, 1997.

\bibitem{Feldkamp1984}
Lee~A. Feldkamp, L.~C. Davis, and James~W. Kress.
\newblock {Practical cone-beam algorithm}.
\newblock {\em Journal of the Optical Society of America A}, 1(6):612, jun
  1984.

\bibitem{Flohr2003}
Thomas Flohr, Karl Stierstorfer, Herbert Bruder, J.~Simon, Arkadiusz Polacin,
  and Stefan Schaller.
\newblock {Image reconstruction and image quality evaluation for a 16-slice CT
  scanner}.
\newblock {\em Medical Physics}, 30(5):832--845, apr 2003.

\bibitem{Tang2006}
Xiangyang Tang, Jiang Hsieh, Roy~A. Nilsen, Sandeep Dutta, Dmitry Samsonov, and
  Akira Hagiwara.
\newblock {A three-dimensional-weighted cone beam filtered backprojection
  (CB-FBP) algorithm for image reconstruction in volumetric CT-helical
  scanning}.
\newblock {\em Physics in Medicine and Biology}, 51(4):855--874, feb 2006.

\bibitem{Gilbert1972}
Peter Gilbert.
\newblock {Iterative methods for the three-dimensional reconstruction of an
  object from projections}.
\newblock {\em Journal of Theoretical Biology}, 36(1):105--117, jul 1972.

\bibitem{Agulleiro2010}
Jose.~I. Agulleiro, Eduardo.~M. Garz{\'{o}}n, Inmaculada Garc{\'{i}}a, and
  J.~J. Fern{\'{a}}ndez.
\newblock {Vectorization with SIMD extensions speeds up reconstruction in
  electron tomography}.
\newblock {\em Journal of Structural Biology}, 170(3):570--575, 2010.

\bibitem{Spitzbarth2015}
Martin Spitzbarth and Malte Drescher.
\newblock {Simultaneous iterative reconstruction technique software for
  spectral-spatial EPR imaging}.
\newblock {\em Journal of Magnetic Resonance}, 257:79--88, aug 2015.

\bibitem{Goris2012}
Bart Goris, Wouter {Van den Broek}, Kees~Joost Batenburg, Hamed {Heidari
  Mezerji}, and Sara Bals.
\newblock {Electron tomography based on a total variation minimization
  reconstruction technique}.
\newblock {\em Ultramicroscopy}, 113:120--130, feb 2012.

\bibitem{Chen2013}
Zhiqiang Chen, Xin Jin, Liang Li, and Ge~Wang.
\newblock {A limited-angle CT reconstruction method based on anisotropic TV
  minimization}.
\newblock {\em Physics in medicine and biology}, 58(7):2119--41, apr 2013.

\bibitem{Gu2017}
Jawook Gu and Jong~Chul Ye.
\newblock {Multi-Scale Wavelet Domain Residual Learning for Limited-Angle CT
  Reconstruction}.
\newblock {\em arXiv submission}, 2017.

\bibitem{Hammernik2017}
Kerstin Hammernik, Tobias W{\"{u}}rfl, Thomas Pock, and Andreas Maier.
\newblock {A Deep Learning Architecture for Limited-Angle Computed Tomography
  Reconstruction}.
\newblock In {\em Bildverarbeitung f{\"{u}}r die Medizin}, pages 92--97.
  Springer Vieweg, Berlin, Heidelberg, 2017.

\bibitem{Kostler2006}
Harald K{\"{o}}stler, Michael Pr{\"{u}}mmer, Ulrich R{\"{u}}de, and Joachim
  Hornegger.
\newblock {Adaptive variational sinogram interpolation of sparsely sampled CT
  data}.
\newblock In {\em Proceedings - International Conference on Pattern
  Recognition}, volume~3, pages 778--781. IEEE, 2006.

\bibitem{Zhang2011}
Yi~Zhang, Yi~Fei Pu, Jin~Rong Hu, Yan Liu, and Ji~Liu Zhou.
\newblock {A new CT metal artifacts reduction algorithm based on
  fractional-order sinogram inpainting}.
\newblock {\em Journal of X-Ray Science and Technology}, 19(3):373--384, 2011.

\bibitem{Li2014}
Si~Li, Qing Cao, Yang Chen, Yining Hu, Limin Luo, and Christine Toumoulin.
\newblock {Dictionary learning based sinogram inpainting for CT sparse
  reconstruction}.
\newblock {\em Optik}, 125(12):2862--2867, 2014.

\bibitem{Gu2006}
Jianwei Gu, Li~Zhang, Guoqiang Yu, Yuxiang Xing, and Zhiqiang Chen.
\newblock {X-ray CT metal artifacts reduction through curvature based sinogram
  inpainting}.
\newblock {\em Journal of X-ray Science and Technology}, 14(2):73--82, 2006.

\bibitem{Zhang2017}
Hanming Zhang, Linyuan Wang, Yuping Duan, Lei Li, Guoen Hu, and Bin Yan.
\newblock {Euler's Elastica Strategy for Limited-Angle Computed Tomography
  Image Reconstruction}.
\newblock {\em IEEE Transactions on Nuclear Science}, 64(8):2395--2405, 2017.

\bibitem{Kalke2014}
Martti Kalke and Samuli Siltanen.
\newblock {Sinogram Interpolation Method for Sparse-Angle Tomography}.
\newblock {\em Applied Mathematics}, 05(03):423--441, 2014.

\bibitem{Thirion1991}
Jean-Philippe Thirion.
\newblock {A Geometric Alternative to Computed Tomography}.
\newblock In {\em Engineering in Medicine and Biology Society}, volume~13,
  page~34, 1991.

\bibitem{Burger2014}
Martin Burger, Jahn M{\"{u}}ller, Evangelos Papoutsellis, and Carola-Bibiane
  Sch{\"{o}}nlieb.
\newblock {Total Variation Regularisation in Measurement and Image Space for
  PET reconstruction}.
\newblock {\em Inverse Problems}, 30(10), oct 2014.

\bibitem{Dong2013}
Bin Dong, Jia Li, and Zuowei Shen.
\newblock {X-ray CT image reconstruction via wavelet frame based regularization
  and Radon domain inpainting}.
\newblock {\em Journal of Scientific Computing}, 54(2-3):333--349, feb 2013.

\bibitem{Ochs2017}
Peter Ochs, Jalal~M. Fadili, and Thomas Brox.
\newblock {Non-smooth Non-convex Bregman Minimization: Unification and new
  Algorithms}.
\newblock {\em arXiv submission}, 2017.

\bibitem{Beck2009}
Amir Beck and Marc Teboulle.
\newblock {Fast gradient-based algorithms for constrained total variation image
  denoising and deblurring problems}.
\newblock {\em IEEE Transactions on Image Processing}, 18(11):2419--2434, 2009.

\bibitem{Chan2006}
Tony~F. Chan, Selim Esedoglu, and Mila Nikolova.
\newblock {Algorithms for Finding Global Minimizers of Image Segmentation and
  Denoising Models}.
\newblock {\em SIAM Journal on Applied Mathematics}, 66(5):1632--1648, 2006.

\bibitem{Benning2017}
Martin Benning, Michael M{\"{o}}ller, Raz~Z. Nossek, Martin Burger, Daniel
  Cremers, Guy Gilboa, and Carola-Bibiane Sch{\"{o}}nlieb.
\newblock {Nonlinear spectral image fusion}.
\newblock In {\em Scale Space and Variational Methods in Computer Vision},
  volume LNCS 10302, pages 41--53, mar 2017.

\bibitem{Weickert1998}
Joachim Weickert.
\newblock {Anisotropic diffusion in image processing}.
\newblock {\em Image Rochester NY}, 256(3):170, 1998.

\bibitem{Berkels2006}
Benjamin Berkels, Martin Burger, Marc Droske, Oliver Nemitz, and Martin Rumpf.
\newblock {Cartoon extraction based on anisotropic image classification}.
\newblock In {\em Vision, Modeling, and Visualization}, page 293. IOS Press,
  2006.

\bibitem{Estellers2015}
Virginia Estellers, Stefano Soatto, and Xavier Bresson.
\newblock {Adaptive Regularization With the Structure Tensor}.
\newblock {\em IEEE Transactions on Image Processing}, 24(6):1777--1790, jun
  2015.

\bibitem{Bertalmio2000}
Marcelo Bertalmio, Guillermo Sapiro, Vincent Caselles, and Coloma Ballester.
\newblock {Image inpainting}.
\newblock In {\em Proceedings of the 27th annual conference on Computer
  graphics and interactive techniques - SIGGRAPH '00}, pages 417--424, New
  York, New York, USA, 2000. ACM Press.

\bibitem{Chambolle2011}
Antonin Chambolle and Thomas Pock.
\newblock {A first-order primal-dual algorithm for convex problems with
  applications to imaging}.
\newblock {\em Journal of Mathematical Imaging and Vision}, 40(1):120--145, may
  2011.

\bibitem{Mosek}
{Mosek ApS}.
\newblock {The Mosek optimization software}.
\newblock {\em Online at http://www.mosek.com}, 2010.

\bibitem{Drusvyatskiy2016}
Dmitriy Drusvyatskiy and Adrian~S. Lewis.
\newblock {Error bounds, quadratic growth, and linear convergence of proximal
  methods}.
\newblock {\em arXiv submission}, page~23, 2016.

\bibitem{iPalm}
Thomas Pock and Shoham Sabach.
\newblock {Inertial Proximal Alternating Linearized Minimization (iPALM) for
  Nonconvex and Nonsmooth Problems}.
\newblock {\em SIAM Journal on Imaging Sciences}, 9(4):1756--1787, jan 2016.

\bibitem{Liang2016}
Jingwei Liang, Jalal~M. Fadili, and Gabriel Peyr{\'{e}}.
\newblock {A Multi-step Inertial Forward-Backward Splitting Method for
  Non-convex Optimization}.
\newblock {\em Advances in Neural Information Processing Systems}, (29):1--9,
  2016.

\bibitem{iPiano}
Peter Ochs, Yunjin Chen, Thomas Brox, and Thomas Pock.
\newblock {iPiano: Inertial Proximal Algorithm for Non-Convex Optimization}.
\newblock {\em SIAM Journal on Imaging Sciences}, 7(2):1388--1419, apr 2014.

\bibitem{Palm}
J{\'{e}}r{\^{o}}me Bolte, Shoham Sabach, and Marc Teboulle.
\newblock {Proximal alternating linearized minimization for nonconvex and
  nonsmooth problems}.
\newblock {\em Mathematical Programming}, 146(1-2):459--494, 2014.

\bibitem{Drusvyatskiy2016a}
Dmitriy Drusvyatskiy, Alexander~D. Ioffe, and Adrian~S. Lewis.
\newblock {Nonsmooth optimization using Taylor-like models: error bounds,
  convergence, and termination criteria}.
\newblock {\em arXiv submission}, pages 1--21, 2016.

\bibitem{Duchi2017}
John~C. Duchi.
\newblock {Introductory Lectures on Stochastic Population Systems}.
\newblock {\em arXiv submission}, pages 1--84, may 2017.

\bibitem{Collins2015}
Sean~M. Collins, Emilie Ringe, Martial Duchamp, Zineb Saghi, Rafal~E.
  Dunin-Borkowski, and Paul~A. Midgley.
\newblock {Eigenmode Tomography of Surface Charge Oscillations of Plasmonic
  Nanoparticles by Electron Energy Loss Spectroscopy}.
\newblock {\em ACS Photonics}, 2(11):1628--1635, nov 2015.

\bibitem{CVX}
Michael Grant, Stephen Boyd, and Yinyu Ye.
\newblock {CVX: Matlab software for disciplined convex programming}, 2008.

\bibitem{Fanelli2008}
Duccio Fanelli and Ozan {\"{O}}ktem.
\newblock {Electron tomography: A short overview with an emphasis on the
  absorption potential model for the forward problem}.
\newblock {\em Inverse Problems}, 24(1):013001, feb 2008.

\bibitem{Chambolle1997}
Antonin Chambolle and Pierre-Louis Lions.
\newblock {Image recovery via total variation minimization and related
  problems}.
\newblock {\em Numerische Mathematik}, 76(2):167--188, 1997.

\end{thebibliography}
%	\printbibliography
	\appendix
	\section{Theorem \ref{smooth tensor}}\label{App: smooth tensor}
	\begin{theorem*}
		If
		\begin{enumerate}
			\item $c_i$ are $2k$ times continuously differentiable in $\Delta$ and $\Sigma$, $k\geq1$
			\item $c_1(x|0,\Sigma) = c_2(x|0,\Sigma)$ for all $x$ and $\Sigma\geq0$
			\item $\partial_\Delta^{2j-1} c_1(x|0,\Sigma) = \partial_\Delta^{2j-1} c_2(x|0,\Sigma) = 0$ for all $x$ and $\Sigma\geq0, j=1\ldots,k$
		\end{enumerate}
		Then $A_d$ is $C^{2k-1}$ with respect to $d$ for all $\rho>0,\sigma\geq0$.
	\end{theorem*}
	In this proof we will drop the $x$ argument from $c_i$ for conciseness of notation.	Define
	\[M_d = (\nabla d_\rho\nabla d_\rho^T)_\sigma\]
	Note that convolutions are bounded linear maps and $\nabla d_\rho\in L^2$ by Young's inequality hence $M\colon L^1(\F R^2,\F R)\to L^\infty(\F R^2,\op{Sym}_2)$ is well defined and differentiable w.r.t. $d$. Hence, it suffices to show that $A$ is differentiable w.r.t. $M$ where
	\[M_d = \lambda_1e_1e_1^T+\lambda_2e_2e_2^T, \lambda_1\geq \lambda_2\implies A = c_1(\Delta,\Sigma)e_1e_1^T+c_2(\Delta,\Sigma)e_2e_2^T \]
	where $\Delta=\lambda_1-\lambda_2, \Sigma=\lambda_1+\lambda_2$. Note that this is not a trivial statement, we can envisage very simple cases in which the (ordered) eigenvalue decomposition is not even continuous. For instance \[M(t) = \MTwo{1-t}{0}{0}{t} \implies \Sigma(t) = 1, \Delta(t) = |1-2t|, e_1 = \splitln{(1,0)^T}{t<\sfrac12}{(0,1)^T}{t>\sfrac12}\]
	Hence we can see that despite having $M\in C^\infty$ we cannot even guarantee that the decomposition is continuous and so cannot employ any chain rule to say that $A$ is smooth.
	
	The structure of this proof breaks into 4 parts:
	\begin{enumerate}
		\item If $c_1(0,\Sigma) = c_2(0,\Sigma)$ then $A$ is well defined
		\item If $c_i\in C^2$ for some open neighbourhood of point $x$ such that $\lambda_1(x)>\lambda_2(x)$ then $A$ is differentiable w.r.t. $M$ on an open neighbourhood of $x$
		\item If $\partial_\Delta c_1(0,\Sigma) = \partial_\Delta c_2(0,\Sigma) = 0$ at a point, $x$, where $\lambda_1(x)=\lambda_2(x)$ then $A$ is differentiable on an open neighbourhood of $x$
		\item If $\partial_\Delta^{2j-1} c_1(0,\Sigma) = \partial_\Delta^{2j-1} c_2(0,\Sigma) = 0$ at a point $x$ where $\lambda_1(x)=\lambda_2(x)$ and for all $j=1\ldots,k$ then $A$ is $C^{2k-1}$ on an open neighbourhood of $x$
	\end{enumerate}
	\begin{proof}\hfill
	\\Proof of part i: Note that when $\lambda_1=\lambda_2$ the choice of $e_i$ is non-unique subject to $e_1e_1^T + e_2e_2^T = \op{id}$ and so we get \[A = c_1(0,\Sigma)\op{id} + (c_2(0,\Sigma)-c_1(0,\Sigma))e_2e_2^T\]
	Hence $A$ is well defined if and only if $c_1(0,\Sigma) = c_2(0,\Sigma)$ for all $\Sigma\geq0$.
	
	\bigbreak
	As we are decomposing $2\times 2$ matrices it will simplify the proof to write explicit forms for the values under consideration.
	\[M = \MTwo{M_{11}}{M_{12}}{M_{12}}{M_{22}}\implies \lambda_i = \frac{M_{11}+M_{22} \pm \sqrt{(M_{11}-M_{22})^2+4M_{12}^2}}{2}\]
	\[\Sigma = M_{11} + M_{22}, \quad \Delta = \sqrt{(M_{11}-M_{22})^2+4M_{12}^2}\]
	\[\Delta\neq 0 \implies e_1 = \frac{(2M_{12},\Delta-M_{11}+M_{22})^T}{\sqrt{(\Delta-M_{11}+M_{22})^2+4M_{12}^2}}=\frac{(\Delta+M_{11}-M_{22},2M_{12})^T}{\sqrt{(\Delta+M_{11}-M_{22})^2+4M_{12}^2}},\]\[ \quad e_2 = \MTwo{0}{-1}{1}{0}e_1\]
	We give two equations for $e_1$ to account for the case when we get $\frac{(0,0)^T}{0}$.
	\bigbreak Proof of part ii: Note that $\Sigma$ is always smooth and $\Delta$ is smooth on the set $\{\Delta>0\}$
	\\Case $M_{12}(x)\neq 0$: Now both equations of $e_1$ are valid (and equal) and the denominators non-zero on a neighbourhood. Hence, we can apply the standard chain rule and product rule to conclude.
	\\Case $M_{12}(x)=0$: In this case $M(x)$ is diagonal but as $\Delta = |M_{11}-M_{22}|>0$ we know that one of our formulae for $e_1$ must be valid with non-vanishing denominator in a neighbourhood and so we can conclude as in the first case.
	\bigbreak Proof of part iii: Given the Neumann condition on $c_i$ we shall express $c_i$ locally by Taylor's theorem.
	\[c_i(\Delta,\Sigma) = c_i(0,\Sigma) + O(\Delta^2) = c_1(0,\Sigma)+O(\Delta^2)\]
	Now we consider a perturbation:
	\[M = \MTwo{m}{0}{0}{m}, \quad \epsilon = \MTwo{\epsilon_{11}}{\epsilon_{12}}{\epsilon_{12}}{\epsilon_{22}}\]
	\[\implies A(M+\epsilon)-A(M) = (c_1(0,2m+\epsilon_{11}+\epsilon_{22})-c_1(0,2m))\op{id} +O(\Delta^2)\]
	\[\Delta^2 = (\epsilon_{11}-\epsilon_{22})^2+ 4\epsilon_{12}^2 = O(\norm{\epsilon}^2)\implies O(\Delta^2)\leq O(\norm{\epsilon}^2)\]
	\[\implies \frac{A(M+\epsilon)-A(M)}{\norm{\epsilon}} = \frac{\partial_\Sigma c_1(0,2m)\op{tr}(\epsilon)}{\norm{\epsilon}} + O(\norm{\epsilon})\]
	In particular, $A$ is differentiable w.r.t. $M$ at $x$.
	\bigbreak Proof of part iv: The proof of this follows exactly as the previous part,
	\[c_i(\Delta,\Sigma) = \sum_0^{k-1} \frac{\Delta^{2j}}{j!}\partial_\Delta^{2j}c_i(0,\Sigma) + O(\Delta^{2k})\]
	where the remainder term is sufficiently smooth. Hence $c_i$ is at least $2k-1$ times differentiable w.r.t. $M$
	\end{proof}
	\section{Theorem \ref{conv thm}}\label{App: conv thm}
	\begin{theorem*}
		If 
		\begin{itemize}
			\item $c_i$ are bounded away from 0
			\item $\rho>0$
			\item $A_d$ is differentiable in $d$
		\end{itemize}
		then sub-level sets of $E$ are weakly compact in $L^2(\Omega,\F R)\times L^2(\F R^2,\F R)$ and $E$ is weakly lower semi-continuous. i.e. for all $(u_n,v_n)\in L^2(\Omega,\F R)\times L^2(\F R^2,\F R)$:
		\[E(u_n,v_n) \text{ uniformly bounded implies a subsequence converges weakly} \]
		\[\liminf E(u_n,v_n) \geq E(u,v) \text{ whenever } u_n\rightharpoonup u, v_n\rightharpoonup v\]
	\end{theorem*}
	\begin{proof}
	If $c_i$ are bounded away from 0 then in particular we have $A_{\C Ru_n}\gtrsim1$ so $\DTV_{u_n}(v_n) = \norm{A_{\C Ru_n}\nabla v_n} \gtrsim \norm{\nabla v_n} = \TV(v_n)$. Hence, 
	\begin{equation*}\begin{split}
	E(u_n,v_n) \text{ uniformly bounded } & \implies \\ &\norm{S_{\Omega'^c}(\C Ru_n-v_n)}_2^2 + \norm{S_{\Omega'}\C Ru_n-b}_2^2 + \norm{S_{\Omega'}v_n-b}_2^2 
	\\&\qquad+ \TV(u_n) + \TV(v_n) \text{ uniformly bounded} 
	\\&\implies \norm{A(u,v)^T - b}_2^2 + \TV\left((u,v)\right) \text{ uniformly bounded} 
	\end{split}\end{equation*}
	for some linear $A$ and constant $b$. Thus we are in a very classical setting where weak compactness can be shown in both the $\norm{(u,v)}_2$ norm and $\norm{(u,v)}_1+\TV((u,v))$ \cite{Chambolle1997}.
	
	\bigbreak We now proceed to the second claim, that $E$ is weakly lower semi-continuous. Note that all of the convex terms in our energy are lower semi-continuous by classical arguments so it remains to show that the non-convex term is too. i.e. 
	\[(u_n,v_n)\rightharpoonup (u,v)\stackrel{?}{\implies} \liminf\norm{A_{\C Ru_n}\nabla v_n}_{2,1} \geq \norm{A_{\C Ru}\nabla v}_{2,1}\]
	We shall first present a lemma from distribution theory.
	\begin{lemma}
		If $\Omega$ is compact, $\varphi\in C^\infty(\Omega,\F R)$ and $w_n\stackrel{L^p}{\rightharpoonup} w$ then 
		\[w_n\star \varphi \to w\star \varphi \text{ convergence in } C^k(\Omega,\F R) \text{ for all } k<\infty\]
	\end{lemma}
	\begin{proof}
		Recall that $w_n\rightharpoonup w\implies \norm{w_n}_p\leq W $ for some $W<\infty$. By H\"older's inequality:
		\[|w_n\star\varphi (x) - w\star\varphi (y)| \leq \int |w_n(z)||\varphi(x-z)-\varphi(y-z)| \lesssim_{p,\Omega} |x-y|W\norm{\varphi'}_\infty\]
		\[|w_n\star\varphi(x)|\lesssim_{p,\Omega} W\norm{\varphi}_\infty \]
		i.e. $\{w_n\st n\in \F N\}$ is uniformly bounded and uniformly (Lipschitz) continuous.
		\[w_n\rightharpoonup w \implies w_n\star \varphi(x)-w\star\varphi(x) = \IP{w_n-w}{\varphi(x-\cdot)} \to 0 \text{ for every } x\]
		Hence, we also know $w_n\star \varphi$ converges point-wise to a unique continuous function. Suppose $\norm{w_{n_k}\star \varphi-w\star \varphi}_\infty \geq \epsilon>0$ for some $\epsilon$ and sub-sequence $n_k\to\infty$. By the Arzela-Ascoli theorem some further subsequence must converge uniformly to the point-wise limit, $w\star \varphi$, which gives us the required contradiction. Hence, $w_n\star \varphi\to w\star \varphi$ in $L^\infty = C^0$. The general theorem follows by induction.
	\end{proof}
	This lemma gives us two direct results.
	\[\rho>0\implies (Ru_n)_\rho \to (Ru)_\rho\text{ in }L^\infty\]
	\[\{(Ru_n)_\rho\}\union\{(Ru)_\rho\} \text{ compact, } A_d\in C^1(d) \implies A_{Ru_n} \to A_{Ru} \text{ in } \norm{\cdot}_{2,\infty}\]
	Hence, whenever $w\in W^{1,1}$ we have
	\begin{align*}
	\norm{A_{\C Ru_n}\nabla w} &\geq \norm{A_{\C Ru}\nabla w} - \norm{(A_{\C Ru_n}-A_{\C Ru})\nabla w} \notag
	\\&\geq \norm{A_{\C Ru}\nabla w}-\norm{A_{\C Ru_n}-A_{\C Ru}}_{2,\infty}\TV(w)
	\end{align*}
	By density of $W^{1,1}$ in the space of Bounded Variation, we know the same holds for $w=v_n$. We also know $\TV(v_n)$ is uniformly bounded thus \[\liminf \norm{A_{\C Ru_n}\nabla v_n} = \liminf \norm{A_{\C Ru}\nabla v_n}\]
	Hence, for all $\norm{\varphi}_{2,\infty}\leq 1$ we have
	\[\IP{v}{\nabla \cdot(A_{\C Ru}\varphi)} = \liminf_n \IP{v_n}{\nabla \cdot(A_{\C Ru}\varphi)} \leq \liminf \norm{A_{\C Ru}\nabla v_n} \leq \liminf \norm{A_{\C Ru_n}\nabla v_n}\]
	as required.
	\end{proof}
\end{document}

\title{Influence of Hyper-Parameters on Reconstruction}
\author{}
\maketitle

As has been noted in the main text, there are many hyper-parameters to tune for the best reconstruction. We commonly found that reconstructions were qualitatively insensitive near the optimal parameter choice but we include here some illustrations of the typical effect of each parameter. The full model is 
$$ E(u,v) = \frac{1}{2}\norm{\C Ru-v}^2_{\alpha_1} + \frac{\alpha_2}{2}\norm{S\C Ru-b}_2^2
+ \frac{\alpha_3}{2}\norm{Sv-b}_2^2 + \beta_1\op{TV}(u) + \beta_2\DTV(v)$$
To remove a degree of equivalence we have always normalised such that $\alpha_2=1$. To construct the directional TV functional we need 2 smoothing parameters, $\rho$ and $\sigma$
\begin{align*}
A_d(x) &\coloneqq c_1(\lambda_1(x),\lambda_2(x))\vec e_1(x)\vec e_1(x)^T + c_2(\lambda_1(x),\lambda_2(x))\vec e_2(x)\vec e_2(x)^T
\\&\text{ such that } (\nabla d_\rho \nabla d_\rho^T)_\sigma(x) = \lambda_1(x)\vec e_1(x)\vec e_1(x)^T + \lambda_2(x)\vec e_2(x)\vec e_2(x)^T
\\&\qquad\qquad\quad \lambda_1(x)\geq \lambda_2(x)\geq0
\end{align*}
Again, we kept $\rho=1$ fixed and only show reconstructions for different values of $\sigma$. The optimal parameters for the Shepp-Logan phantom referred to below were
$$\alpha_1 =\frac1{4^2}\1_{\Omega'^c}, \alpha_3 = 3\times10^{-1}, \beta_1=3\times10^{-5}, \beta_2=3\times10^{2}, \beta_3=10^{10}, \sigma=8 $$

\clearpage
\begin{figure}
	\includegraphics[width=\linewidth,trim={180 75 125 45},clip]{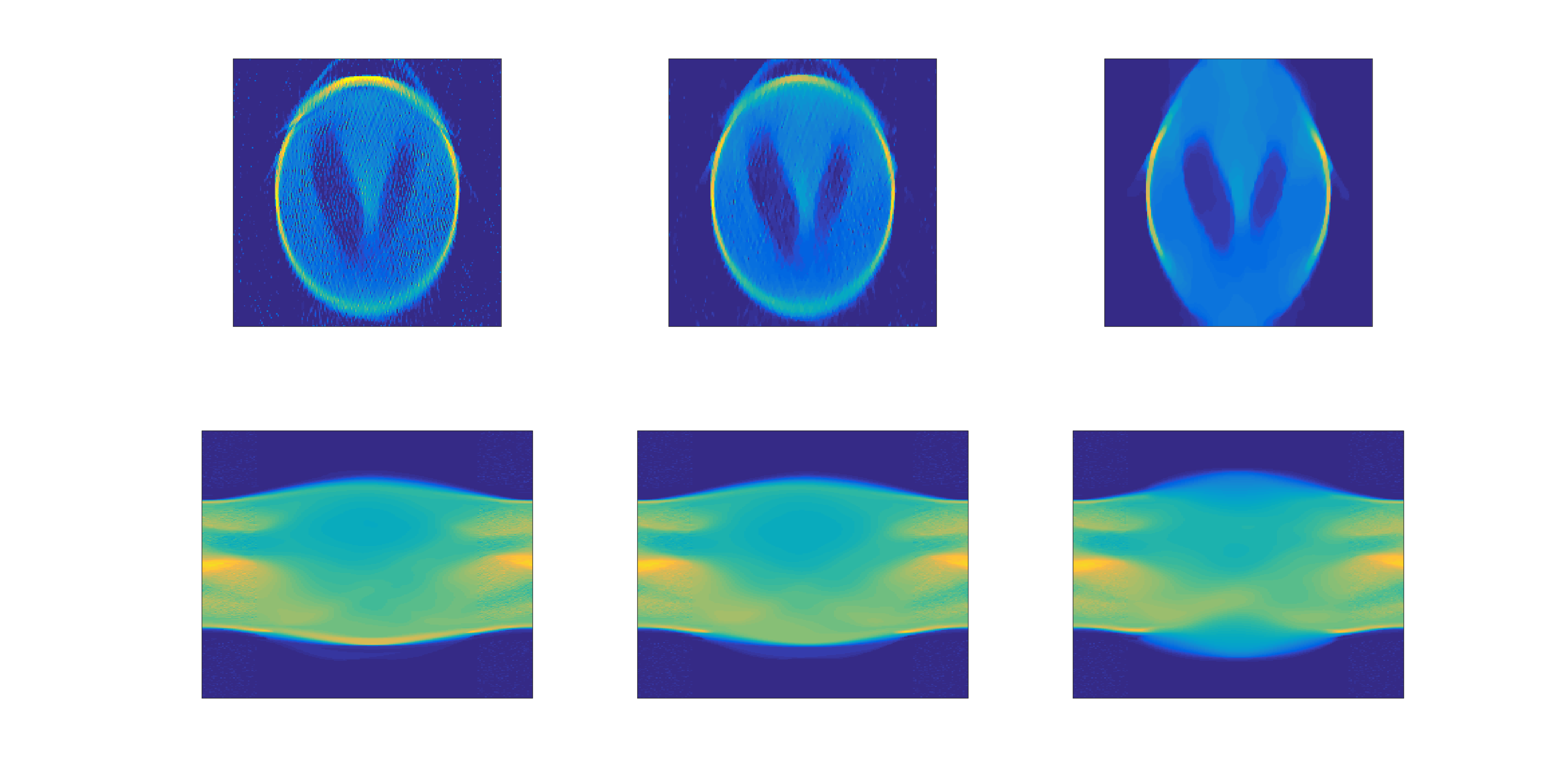}
	\caption{Varying reconstruction for low (first column), optimal (middle column) and high (right column) values of $\beta_1$ (TV regularisation parameter). `low' is a factor of 0.1 lower than `optimal' and `high' a factor of 10 higher.} \label{param fig1}
\end{figure}
\begin{figure}
	\includegraphics[width=\linewidth,trim={180 75 125 45},clip]{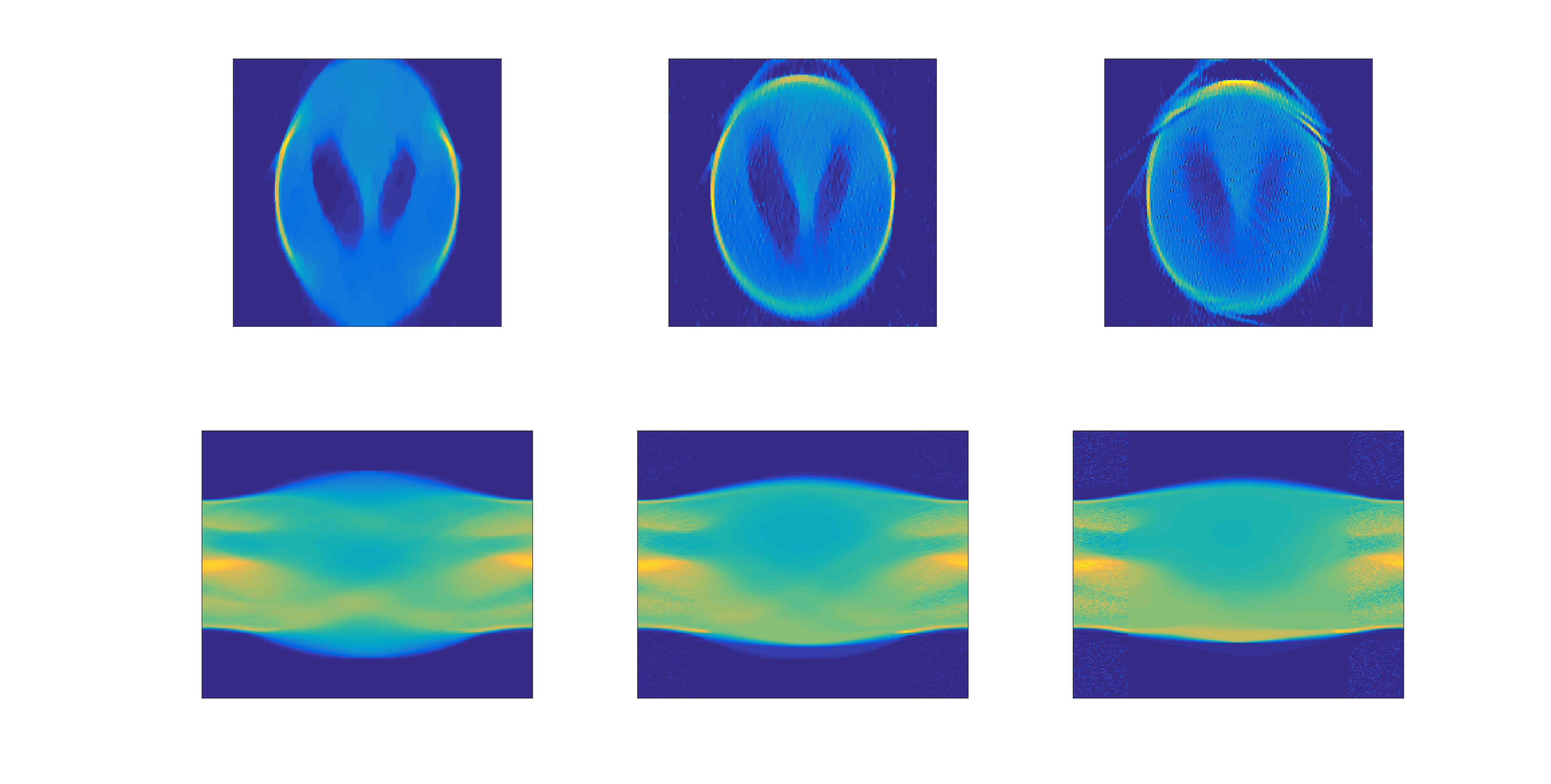}
	\caption{Varying reconstruction for low (first column), optimal (middle column) and high (right column) values of $\beta_2$ (DTV regularisation parameter). `low' is a factor of 0.1 lower than `optimal' and `high' a factor of 10 higher.}
\end{figure}
\begin{figure}
	\includegraphics[width=\linewidth,trim={180 75 125 45},clip]{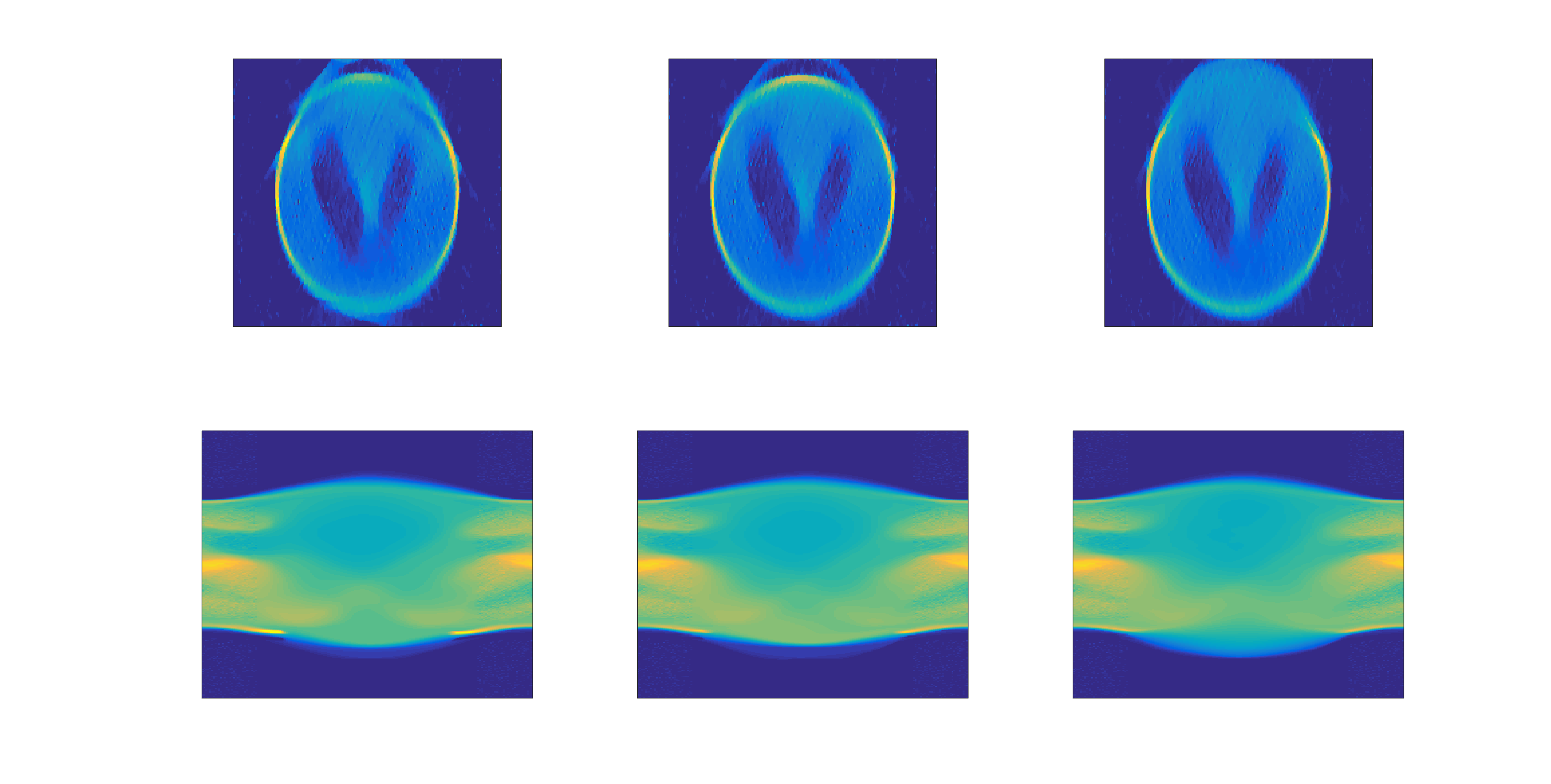}
	\caption{Varying reconstruction for low (first column), optimal (middle column) and high (right column) values of $\alpha_1$ (pairing term between $u$ and $v$). `low' is a factor of 0.1 lower than `optimal' and `high' a factor of 10 higher.}
\end{figure}
\begin{figure}
	\includegraphics[width=\linewidth,trim={180 75 125 45},clip]{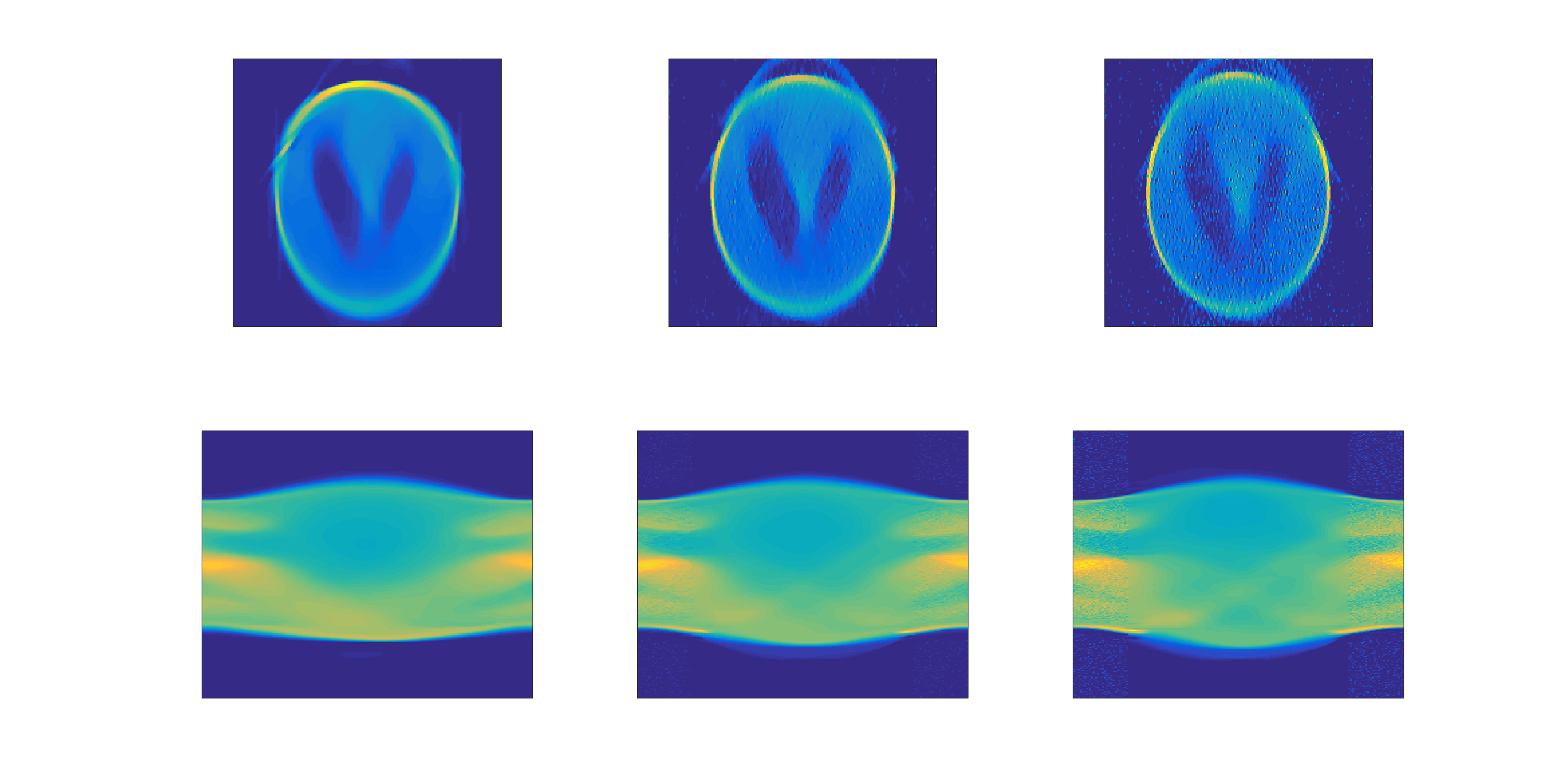}
	\caption{Varying reconstruction for low (first column), optimal (middle column) and high (right column) values of $\alpha_3$ (sinogram noise parameter). `low' is a factor of 0.1 lower than `optimal' and `high' a factor of 10 higher.}
\end{figure}
\begin{figure}
	\includegraphics[width=\linewidth,trim={180 75 125 45},clip]{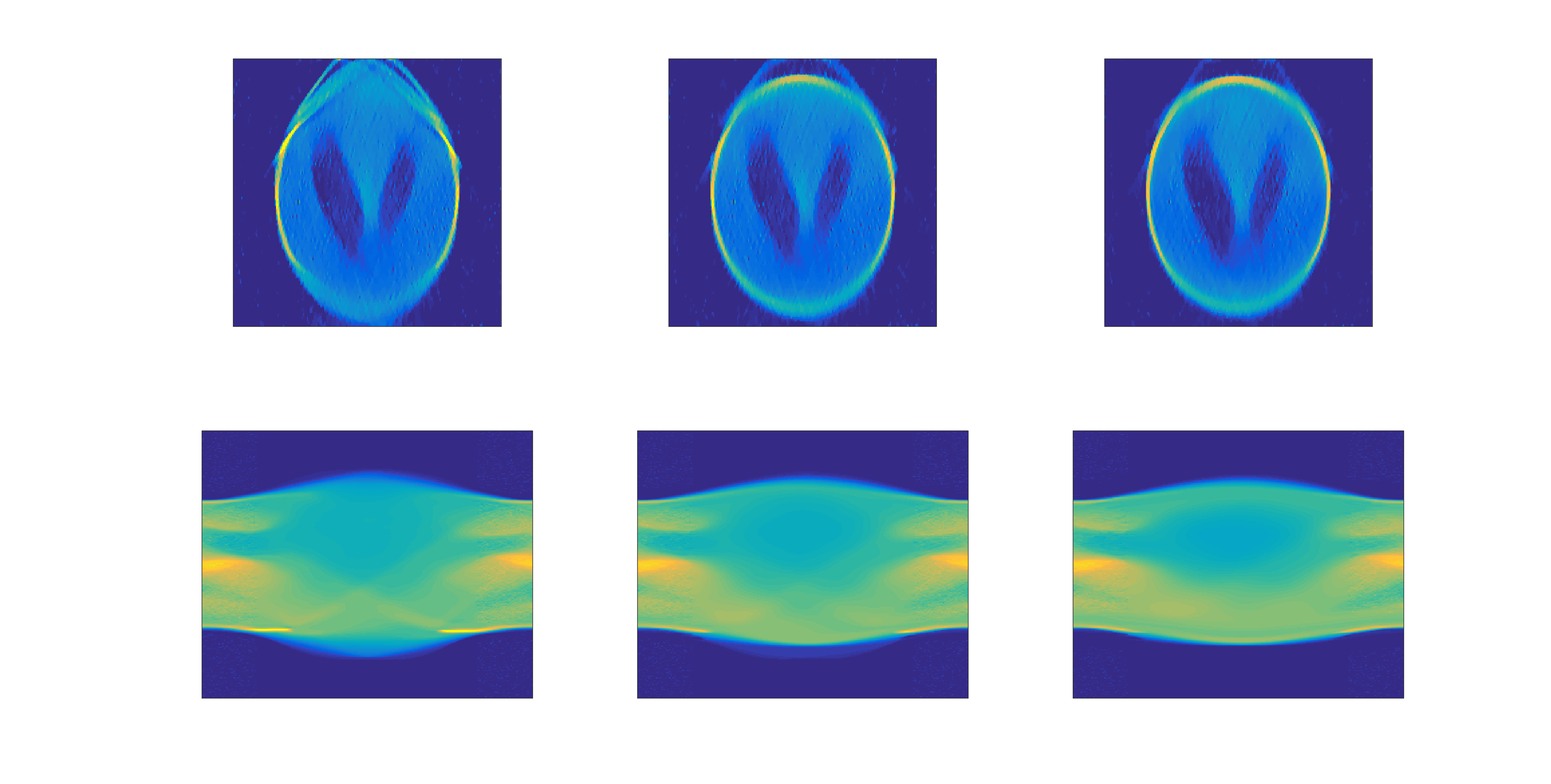}
	\caption{Varying reconstruction for low (first column), optimal (middle column) and high (right column) values of $\sigma$ (smoothing parameter inside DTV functional). `low' is a factor of 0.5 lower than `optimal' and `high' a factor of 2 higher.}\label{param figEnd}
\end{figure}